\documentclass[aap]{imsart}

\RequirePackage{amsthm,amsmath,amsfonts,amssymb}
\RequirePackage[numbers]{natbib}
\RequirePackage[colorlinks,citecolor=blue,urlcolor=blue]{hyperref}
\usepackage{enumerate}
\usepackage{graphicx}
\usepackage{comment}
\usepackage{oands}
\usepackage{tikz}
\usepackage{changepage}
\usepackage{bbm}
\usepackage{mathtools}
\usepackage[pagewise,mathlines]{lineno}
\usepackage{appendix}
\usepackage{stmaryrd}
\usepackage{multicol}
\usepackage{microtype}
\usepackage[colorinlistoftodos]{todonotes}
\usepackage{dsfont}
\usepackage{mathrsfs}  
\usepackage{multirow}
\usepackage{graphicx,subfigure}
\usepackage{array}
\usepackage{bm}
\newcolumntype{P}[1]{>{\centering\arraybackslash}p{#1}}

\startlocaldefs
\theoremstyle{plain}
\newtheorem{thm}{Theorem}[section]
\newtheorem{cor}[thm]{Corollary}
\newtheorem{lem}[thm]{Lemma}
\newtheorem{prop}[thm]{Proposition}

\newtheorem{remark}[thm]{Remark}

\newcommand{\dsb}{\begin{adjustwidth}{2.5em}{0pt}
		\begin{footnotesize}}
		\newcommand{\dse}{\end{footnotesize}
\end{adjustwidth}}

\newcommand{\ssb}{\begin{adjustwidth}{2.5em}{0pt}}
	\newcommand{\sse}{\end{adjustwidth}}

\newcommand{\aryb}{\begin{eqnarray*}}
	\newcommand{\arye}{\end{eqnarray*}}
\def\alb#1\ale{\begin{align*}#1\end{align*}}
\def\allb#1\alle{\begin{align}#1\end{align}}
\newcommand{\eqb}{\begin{equation}}
	\newcommand{\eqe}{\end{equation}}
\newcommand{\eqbn}{\begin{equation*}}
	\newcommand{\eqen}{\end{equation*}}

\newcommand{\BB}{\mathbb}

\newcommand{\mcl}{\mathcal}

\newcommand{\ind}{\alpha}

\DeclareMathOperator{\Perm}{Perm}

\DeclareMathOperator{\occ}{occ}
\newcommand{\pocc}{\widetilde{\occ}}
\DeclareMathOperator{\cocc}{c\text{-}occ}
\DeclareMathOperator{\pcocc}{\widetilde{c\text{-}occ}}
\newcommand{\bbP}{\mathbb{P}}
\newcommand{\bbE}{\mathbb{E}}

\newcommand{\bbR}{\mathbb{R}}

\let\originalleft\left
\let\originalright\right
\renewcommand{\left}{\mathopen{}\mathclose\bgroup\originalleft}
\renewcommand{\right}{\aftergroup\egroup\originalright}

\DeclareMathOperator{\Leb}{Leb}

\DeclareMathOperator{\pat}{pat}
\DeclareMathOperator{\Luce}{Luce}
\DeclareMathOperator{\Var}{Var}

\newcommand{\dd}{\text{d}}

\usepackage[capitalise]{cleveref}
\Crefname{thm}{Theorem}{Theorems}
\Crefname{lem}{Lemma}{Lemmas}
\Crefname{cor}{Corollary}{Corollaries}
\Crefname{prop}{Proposition}{Propositions}
\Crefname{conj}{Conjecture}{Conjectures}
\Crefname{defn}{Definition}{Definitions}
\Crefname{remark}{Remark}{Remarks}
\Crefname{ques}{Question}{Questions}
\Crefname{prob}{Problem}{Problems}

\endlocaldefs

\begin{document}

\begin{frontmatter}
\title{Permuton and local limits for the Luce model}
\runtitle{Permuton and local limits for the Luce model}

\begin{aug}
\author[A]{\fnms{Jacopo}~\snm{Borga}\ead[label=e1]{jborga@mit.edu}},
\author[B]{\fnms{Sourav}~\snm{Chatterjee}\ead[label=e2]{souravc@stanford.edu}}
\and
\author[B]{\fnms{Persi}~\snm{Diaconis}\ead[label=e3]{diaconis@math.stanford.edu}}
\address[A]{MIT \printead[presep={,\ }]{e1}}

\address[B]{Stanford University \printead[presep={,\ }]{e2,e3}}
\end{aug}

\begin{abstract}
	We investigate the asymptotic properties of permutations drawn from the Luce model, a natural probabilistic framework in which permutations are generated sequentially by sampling without replacement, with selection probabilities proportional to prescribed positive weights. These permutations arise in applications such as ranking models, the Tsetlin library, and related Markov processes.
	Under minimal assumptions on the weights, we establish a permuton limit theorem describing the global behavior of Luce-distributed permutations and derive an explicit density of the limiting permuton. We further compute limiting pattern densities and analyze the differences between exact Luce permutations and their permuton approximations. We also study the local convergence of these permutations, proving a quenched Benjamini--Schramm limit and a central limit theorem for consecutive pattern occurrences. Finally, we prove a central limit theorem for the number of inversions.
\end{abstract}



\end{frontmatter}


\section{Introduction}

\subsection{The Luce model on permutations}

We consider the Luce model on permutations defined as follows. Fix some positive real weights $\theta_1,\dots,\theta_n$ with total sum $w_n=\theta_1+\dots+\theta_n$ and consider the function
\begin{equation}\label{eq:p-function}
	p\left(\theta_1,\dots,\theta_n\right):=\frac{\theta_1}{w_n}\times\frac{\theta_2}{w_n-\theta_1}\times\frac{\theta_3}{w_n-\theta_1-\theta_2}\times\cdots\times\frac{\theta_n}{\theta_n}.
\end{equation}
Note that the weights $\theta_i$ may depend on $n$, and could therefore be written as $\theta_i^n$. For simplicity, we will suppress this dependence in our notation and write simply $\theta_i$ throughout. For convenience, we also always fix $\theta_0=0$.
We consider the following distribution, called the \textbf{Luce distribution}, on the set $\mcl S_n$ of permutations of $\llbracket n \rrbracket=\{1,2,\dots,n\}$: for all $\sigma_n\in\mcl S_n$,
\begin{equation}\label{eq:defn-Luce}
	\BB P(\sigma_n)=p\left(\theta_{\sigma_n^{-1}(1)},\dots,\theta_{\sigma_n^{-1}(n)}\right).
\end{equation}
Sometimes, we write $\Luce(\theta_1,\dots,\theta_n)$ for the law of a random Luce-distributed permutation of size $n$ with weights $\theta_1,\dots,\theta_n$. Note that multiplying all weights by a fixed constant does not change the Luce distribution. We will therefore freely rescale when convenient (e.g., replacing $\theta_i$ by $\theta_i/n$, or normalizing the $\theta_i$ so that $\sum_i\theta_i=n$ or $1$).


\paragraph{Sampling and clarification on inverses.}
One can sample a Luce-distributed permutation $\sigma_n$ as follows.  
Place $n$ numbered balls with weights $\theta_1,\dots,\theta_n$ into an urn.  
At each step, withdraw one ball (sampling without replacement) with probability proportional to its weight relative to the remaining balls. Record the order in which the balls are withdrawn, which we denote (in one-line notation; see below for more details on notation) by
\[
\tau = \tau_1\,\tau_2\,\dots\,\tau_n,
\]
where $\tau_j$ is the label of the $j$-th ball drawn.  
From this draw order we obtain a permutation $\sigma\in\mcl S_n$ by declaring
\[
\sigma(\tau_j)=j \quad \text{for all } j=1,\dots,n,
\]
that is, $\sigma(i)$ is the \emph{position} at which label $i$ is drawn.
Equivalently, $\sigma^{-1}(j)=\tau_j$ gives the label drawn at step $j$.  
Thus, in~\eqref{eq:defn-Luce}, the argument of $p$ is simply the sequence of weights in the actual draw order.

\paragraph{Example.}
Suppose we have three balls labeled $1,2,3$ with weights $\theta_1,\theta_2,\theta_3>0$.  
The probability that the draw order is
\[
\tau=312,
\]
i.e.\ ball $3$ is drawn first, then ball $1$, and finally ball $2$, is
\[
\frac{\theta_3}{\theta_1+\theta_2+\theta_3}\;\times\;\frac{\theta_1}{\theta_1+\theta_2}\;\times\;\frac{\theta_2}{\theta_2}
\;=\;\frac{\theta_1\theta_3}{(\theta_1+\theta_2+\theta_3)(\theta_1+\theta_2)}.
\]
This corresponds to the permutation $\sigma$ with
$\sigma(3)=1,\ \sigma(1)=2,\ \sigma(2)=3$, i.e.\ $\sigma=231$.  
Note that evaluating~\eqref{eq:defn-Luce} with
\[
(\theta_{\sigma^{-1}(1)},\theta_{\sigma^{-1}(2)},\theta_{\sigma^{-1}(3)})
= (\theta_3,\theta_1,\theta_2)
\]
yields exactly the same probability.

\begin{remark}\label{rmk:inv}
	Note that our definition in \eqref{eq:defn-Luce} differs from the one provided, for instance, in \cite{chatterjee2023enumerative}. More precisely, if $\sigma\sim\Luce(\theta_1,\dots,\theta_n)$ is defined as in~\eqref{eq:defn-Luce}, then $\sigma^{-1}$ is $\Luce(\theta_1,\dots,\theta_n)$-distributed according to the definition in~\cite{chatterjee2023enumerative}. We made this choice  to simplify some later notation. In this paper, we will consistently use our definition in \eqref{eq:defn-Luce} and translate the results from~\cite{chatterjee2023enumerative} in terms of our definition when we refer to them.
\end{remark}

\cite[Section 2]{chatterjee2023enumerative} provides a literature review. Several applied problems are shown to give rise to the Luce model, ranging from the Tsetlin library problem; the stationary distribution of a Markov chain that, at each step, selects a card labeled $i$ with probability proportional to $\theta_i$ and moves it to the top; certain problems in psychology, as well as applications to poker and horse racing. The main results of~\cite{chatterjee2023enumerative} determine the limiting distribution of $\sigma^{-1}(1)\dots \sigma^{-1}(k)$ (the labels on the
top $k$  cards) and $\sigma^{-1}(n-k)\dots \sigma^{-1}(n)$ (the labels on the
bottom $k$  cards) in a Luce distributed permutation $\sigma$: $\sigma^{-1}(1)\dots \sigma^{-1}(k)$ are approximately i.i.d.\ picks from the weights; $\sigma^{-1}(n-k)\dots \sigma^{-1}(n)$ are very different.

We also recall that one standard choice for the weights $\theta_i$ is the one given by the Sukhatme weights $\theta_i=n-i+1$, see \cite[Section  2.2.4]{chatterjee2023enumerative} for further explanations.

\subsection{Limits of random permutations: local and permuton convergence}

A classical (but somewhat ill-posed) question in the study of random permutations is: 
\begin{equation*}
	\text{\emph{What does a typical large permutation drawn from a given model look like?}}
\end{equation*}
Traditionally, efforts to address this question have focused on the convergence of various permutation statistics, especially in the case of uniformly random permutations. Examples include the number of cycles, the number of inversions, and the length of the longest increasing subsequence, among many others.

In recent years, however, a more geometric perspective has gained interest. Rather than analyzing specific statistics, this approach aims to understand the permutation as a whole -- searching for a global or local description of its ``limiting shape''. This is the point of view adopted in the present work, where we will focus on the case of Luce permutations.

We first introduce the notion of global and local limits for permutations after some reminders on permutation patterns, then in the next subsection we give an overview of our main results.

\bigskip

\noindent\textbf{Notation for patterns in permutations.}
The occurrences of various patterns in a permutation are classical fare, both in applications and in probabilistic combinatorics. For example, the number of inversions, descents, and peaks has been extensively studied. More recently, pattern avoidance has become a hot topic~\cite{bona2008combinatorics22,kitaev2011patterns22,vatter2015permutation23}. We begin by introducing notation for general patterns.

Recall that we denote by $\mcl S_n$ the set of permutations of size $n\geq 1$ and set $\mcl S:=\cup_{n\geq 1}\mcl S_n$.
We sometimes write permutations of $\mcl S_n$ in one-line notation as $\sigma=\sigma(1)\sigma(2)\dots\sigma(n).$ 
If $x_1,\dots ,x_n$ is a sequence of distinct numbers, let $\text{std}(x_1,\dots ,x_n)$ be the unique permutation $\pi$ in $\mcl S_n$ that is in the same relative order as $x_1,\dots ,x_n,$ i.e.\ $\pi(i)<\pi(j)$ if and only if $x_i<x_j.$
Given a permutation $\sigma\in\mcl S_n$ and a subset of indices $I\subset\llbracket n \rrbracket$, let $\pat_I(\sigma)$ be the permutation induced by $(\sigma(i))_{i\in I},$ namely, $\pat_I(\sigma)\coloneqq\text{std}\big((\sigma(i))_{i\in I}\big),$ where the $\sigma(i)$ are listed in increasing value of the index.
For example, if $\sigma=87532461$ and $I=\{2,4,7\}$ then $\pat_{\{2,4,7\}}(87532461)=\text{std}(736)=312$. 

Given two permutations, $\sigma\in\mcl S_n$ and $\pi\in\mcl S_k$ for some $k\leq n,$ and a set of indices $I=\{i_1 < \ldots < i_k\}$, we say that $\sigma(i_1) \ldots \sigma(i_k)$ is an \textbf{occurrence} of $\pi$ in $\sigma$ if $\pat_I(\sigma)=\pi$ (we will also say that $\pi$ is a \textbf{pattern} of $\sigma$). If the indices $i_1, \ldots ,i_k$ form an interval, then we say that $\sigma(i_1) \ldots \sigma(i_k)$ is a \textbf{consecutive occurrence} of $\pi$ in $\sigma$ (we will also say that $\pi$ is a \textbf{consecutive pattern} of $\sigma$). 
We also denote intervals of integers as $\llbracket n,m \rrbracket$ for some $n,m\in\BB Z_{>0}$ with $n\leq m$.
For example, the permutation $\sigma=1532467$ contains an occurrence (but no consecutive occurrence) of $1423$ and a consecutive occurrence of $321$. Indeed, $\pat_{\{1,2,3,5\}}(\sigma)=1423$ but no interval of indices of $\sigma$ induces the permutation $1423.$

We denote by $\occ(\pi,\sigma)$ the number of occurrences of $\pi$ in $\sigma$, that is,
\begin{equation}
	\occ(\pi,\sigma)\coloneqq\#\big\{I \subseteq \llbracket n \rrbracket \big|\#I=k, \pat_I(\sigma)=\pi\big\},
\end{equation}
where $\#$ denotes the cardinality of a set.
Moreover, we denote by $\pocc(\pi,\sigma)$ the proportion of occurrences of $\pi$ in $\sigma,$ that is, $\pocc(\pi,\sigma)\coloneqq\occ(\pi,\sigma)/\binom{n}{k}$.
Similarly, we denote by $\cocc(\pi,\sigma)$ the number of consecutive occurrences of $\pi$ in $\sigma$, that is,
\begin{equation}
	\label{cocc}
	\cocc(\pi,\sigma)\coloneqq\#\Big\{I\subseteq\llbracket n \rrbracket\Big|\#I=k,\;I\text{ is an interval, } \text{pat}_I(\sigma)=\pi\Big\},
\end{equation}
and we denote by $\pcocc(\pi,\sigma)$ the proportion of consecutive occurrences of $\pi$ in $\sigma,$ that is,\footnote{Another natural choice for the denominator would be $n-k+1$. Note that for every fixed $k$ there are no differences in the asymptotics when $n$ tends to infinity.} $\pcocc(\pi,\sigma)\coloneqq\cocc(\pi,\sigma)/n$.

\bigskip

\noindent\textbf{Permuton convergence.} A \textbf{permuton} is a Borel probability measure $\mu$ on the unit square $ [ 0,1 ] ^ 2$ with two uniform marginals, that is, $\mu([a,b]\times[0,1]) = \mu( [0,1] \times [a,b] )=b-a$ for all $0\le a < b\le 1$. Permutons are natural objects to describe the scaling limits of random permutations and have been studied quite intensively in the past decade; see~\cite[Section 2.1]{borga-thesis} or \cite[Section 2]{meliotsurvey} for an introduction to the theory of permutons and, in particular, to the notion of permuton limit.

For a permutation $\sigma\in\mcl{S}_n$, we define the \textbf{permuton associated with $\sigma$} to be the measure $\mu_\sigma$ on $[0,1]^2$ which is equal to $n $ times the Lebesgue measure on
\begin{equation}
	\bigcup_{j =1}^{n} \left[ \frac{j-1}{n} , \frac{j}{n} \right] \times \left[ \frac{\sigma(j)-1}{n} , \frac{\sigma(j)}{n} \right] .
\end{equation}

Let $\sigma_n$ be a random permutation of size $n$. 
Moreover, for any fixed $k\in\mathbb{N}$, let ${ I}_{n,k}$ be a uniform random subset of $\llbracket n \rrbracket$ with $k$ elements, independent of $ \sigma_n$. We recall (see for instance~\cite[Theorem 2.5]{bbfgp-universal}) that the following assertions are equivalent:
\begin{enumerate}[(a)]
	\item  $(\mu_{ \sigma_n})_{n\in\mathbb{N}}$ converges in distribution w.r.t.\ the permuton topology to a random permuton $ \mu$, that is, 
	\[
	\int_{[0,1]^2} f \, d\mu_{ \sigma_n} \rightarrow \int_{[0,1]^2} f \, d\mu,\]
	for every (bounded and) continuous function $f: [0,1]^2 \to \mathbb{R}$.
	\item  The random infinite vector $\big(\pocc(\pi,\sigma_n)\big)_{\pi \in \mathcal{S}}$ converges in distribution w.r.t.\ the product topology to a random infinite vector $( \Lambda(\pi))_{\pi \in \mathcal{S}}$. 
	\item  For every $\pi$ in $\mathcal{S}$, there exists $\Delta_\pi \geq 0$ such that 
	$\mathbb E[\pocc(\pi, \sigma_n)] \xrightarrow{n\to\infty} \Delta_\pi.$
	\item  For every $k\in\mathbb{N}$, the sequence  $\big(\pat_{{ I}_{n,k}}(\sigma_n)\big)_{n\in\mathbb{N}}$ of random permutations converges in distribution to some random permutation $ \rho_k$.
\end{enumerate}	
Whenever these assertions are verified, we have for every $k\in\mathbb{N}$ and $\pi\in\mathcal S_k$,
\[  \mathbb P(\rho_{k} = \pi) = \Delta_\pi = \mathbb E[ \Lambda(\pi)]. 
\] 

\bigskip

The theory of permutons was first developed in~\cite{hoppen2013limits}. Permuton convergence has been studied in various models, including Erdős–Szekeres permutations~\cite{MR2266895}, Mallows permutations~\cite{Shannon}, exponential families on permutations~\cite{Mukherjee}, fixed pattern densities~\cite{Winkler} 
and many others. 
For random pattern-avoiding permutations, the limiting permutons are often random~\cite{miner2014shape, bbfgp-universal,borga-skew-permuton}. In this paper, we further contribute by analyzing the case of Luce-distributed permutations.

\bigskip

\noindent\textbf{Local convergence.} We recall from \cite[Theorem 2.32]{borga2020local} that we say that a sequence of random permutations $(\sigma_n)_n$ (quenched) \textbf{Benjamini--Schramm converges} if one of the following equivalent conditions holds:
\begin{enumerate}[(a)]
	\item There exists a family of non-negative real random variables $(\Gamma^h_{\pi})_{h \geq 1,\pi\in\mathcal{S}_{2h+1}}$ such that for $i_n$ a uniform random index in $\llbracket n \rrbracket$ independent of $\sigma_n$, then
	\[\Big(\BB{P}\big(\pat_{\llbracket i_n-h,i_n+h\rrbracket \cap \llbracket 1,n\rrbracket}(\sigma_n)=\pi\,\big|\,\sigma_n\big)\Big)_{h \geq 1,\pi\in\mathcal{S}_{2h+1}}\xrightarrow[n\to\infty]{\mathrm{d}}\big(\Gamma^h_{\pi}\big)_{h \geq 1,\pi\in\mathcal{S}_{2h+1}},\]
	w.r.t.\ the product topology;
	\item There exists an infinite vector of non-negative real random variables $(\Lambda_{\pi})_{\pi\in\mathcal{S}}$ such that 
	\[\big(\pcocc(\pi,\sigma_n)\big)_{h \geq 1,\pi\in\mathcal{S}_{h}}\xrightarrow[n\to\infty]{\mathrm{d}}\big(\Lambda_{\pi}\big)_{h \geq 1,\pi\in\mathcal{S}_{h}}\] 
	w.r.t.\ the product topology.
\end{enumerate}
In particular, if one of the two conditions holds (and so do both of them), then 
\begin{equation*}
	\Lambda_{\pi}\stackrel{\mathrm{d}}{=}\Gamma^h_{\pi},\quad\text{for all}\quad \pi\in\mathcal{S}_{2h+1}\qquad\text{and}\qquad
	\Lambda_{\pi}\stackrel{\mathrm{d}}{=}\sum_{m=1}^{2h+1}\Gamma^h_{\pi^{*m}},\quad\text{for all}\quad \pi\in\mathcal{S}_{2h},
\end{equation*}
where $\pi^{*m}\coloneqq\text{std}(\pi_1,\dots,\pi_{k},m-1/2)$.

\bigskip

The theory of local convergence for permutations was introduced in~\cite{borga2020local} and subsequently investigated across several models, including those studied in~\cite{bevan2019permutations,MR4115736,borga2021asymptotic,bm-baxter-permutation,park2024galton}.

\subsection{An overview of our main results}

The principal contributions of this work are as follows:

\begin{itemize}
	\item In \Cref{thm:permuton_conv}, we prove a permuton limit result for Luce-distributed permutations (under minimal assumptions on their weights).
	\item In \Cref{Sect:prop-luce}, we investigate properties of the permuton limit of Luce permutations. In particular, in Proposition~\ref{prop:density} we show that this permuton limit is absolutely continuous with respect to the uniform permuton $\Leb_{[0,1]^2}$, and we provide an explicit density in \eqref{eq:density}. In \Cref{sect:patt-dens}, we present explicit computations of limiting pattern densities. Finally, in \Cref{sect:diff-sim}, we discuss the differences and similarities between exact Luce-distributed permutations and permutations sampled from the corresponding Luce permuton limit.
	\item \Cref{sect:local} is dedicated to the study of the local limit of Luce permutations. In \Cref{thm:lln_local}, we establish the quenched Benjamini--Schramm convergence of Luce-distributed permutations, while in \Cref{thm:clt_local}, we prove a stronger result: a central limit theorem for the number of consecutive occurrences of patterns with a Berry--Esseen type error estimate. \Cref{thm:general-local-thm} ensures that the above two results hold under the same minimal assumptions on the weights that we have for the permuton limit. Sections~\ref{sect:app1} and~\ref{sect:app2} provide applications of these results.
	
	\item \Cref{sect:clt-inv} proves a central limit theorem (\Cref{them:clt-inv}) for the number of inversions in a Luce distributed permutation with a Berry--Esseen type error estimate.
	
	\item We conclude the paper with comments and open questions in~\Cref{sect:final}.
\end{itemize}

\paragraph{Acknowledgments.}

We thank  Victor Dubach and Sumit Mukherjee for some helpful discussions and for sharing their interest in this problem. 
J.B.\ was partially supported by NSF grant DMS-2441646. S.C.\ was partially supported by NSF grants DMS-2113242 and DMS-2153654. P.D.\ was partially supported by NSF grant DMS-1959042. 

\section{The permuton limit and its consequences}

This section is organized as follows. \Cref{sect:perm-lim} establishes the permuton limit, \Cref{Sect:prop-luce} studies its properties, \Cref{sect:patt-dens} computes pattern densities, and \Cref{sect:diff-sim} compares Luce permutations with their permuton limit.

\subsection{The permuton limit of Luce-distributed permutations}\label{sect:perm-lim}

\begin{thm}[Permuton limit]\label{thm:permuton_conv}
	Set $f_n(y):=\theta_{\lfloor yn \rfloor}$ for all $y\in[0,1]$ and assume that $f_n\to f$ almost everywhere for some positive, finite, measurable function $f$ on $[0,1]$.
	Let also $\mu$ be the law of the pair
	\begin{equation}\label{eq:limiting_rv}
		(X,Y):=\Big(U,F\big(E/f(U)\big)\Big),
	\end{equation}
	where $U$ is a uniform random variable on $[0,1]$, $E$ is an independent exponential random variable of parameter 1, and
	\begin{equation}\label{eq:defn-F}
		F(x):=1-\int_0^1\mathrm{e}^{-x f(y)}\,\dd y,\quad \text{for all $x\geq 0$.}
	\end{equation}
	If $\sigma_n\sim \Luce(\theta_1,\dots,\theta_n)$, then the sequence $(\sigma_n)_n$ converges in probability in the permuton sense to the deterministic permuton $\mu$.
\end{thm}

For example, if $\theta_i=1$ for all $n\geq 1$ and $i\in\llbracket n \rrbracket$, then $f(x) = 1$ for all $x\in[0,1]$, and hence $Y = 1 - \mathrm{e}^{-E}$.
Thus, $(X, Y )$ is a pair of independent Uniform$([0, 1])$ random variables in this case, as expected.

On the other hand, for the Sukhatme weights, we have $\theta_{i}=(n-i+1)/n$. 
In this case, $f_n(x) \to f(x) = 1 - x$ for all $x \in [0, 1]$. Since $f(x) > 0$ for a.e.\ $x$, we can apply
the theorem and get that
\[
Y=1-\int_0^1\mathrm{e}^{-(1-x)E/(1-X)}\,\dd x=1-\frac{1-\mathrm{e}^{-E/(1-X)}}{E/(1-X)},
\]
where $E$ is a standard exponential random variable, independent of $X$, and $X\sim\text{Uniform}([0, 1])$. Note the (a priori non-trivial) fact that $Y$ is Uniform$([0, 1])$, since $\mu$ is a permuton. See \Cref{fig:mallows_luce} on page~\pageref{fig:mallows_luce} for a picture.


\medskip

We now turn to the proof of the theorem.
Let $\sigma_n\in \mcl S_n$ be a permutation from the Luce model with weights $\theta_{1},\dots,\theta_{n}$. Let $(U^\ind_n)_{\ind\in\llbracket 2 \rrbracket}$ be two i.i.d.\ uniform random variables on $\llbracket n \rrbracket$  independent of $\sigma_n$, and define for  $\ind\in\llbracket 2 \rrbracket=\{1,2\}$,
\[
X^\ind_n:=\frac{U^\ind_n}{n},\qquad Y^\ind_n:=\frac{\sigma_n(U^\ind_n)}{n}.
\]

Then, to prove that $\sigma_n$ converges in probability to $\mu$ in the permuton sense, it is enough to show that 
\begin{equation}\label{eq:fddconv}
	\text{the vector $(X^\ind_n,Y^\ind_n)_{\ind\in\llbracket 2 \rrbracket}$ converges in distribution to $(X^\ind,Y^\ind)_{\ind\in\llbracket 2 \rrbracket}$,}
\end{equation}
where the $(X^1,Y^1)$ and $(X^2,Y^2)$ are i.i.d.\ copies of the random variable in~\eqref{eq:limiting_rv}. Indeed, thanks to \cite[Remark 1.5]{borga2024large}, to show that $\sigma_n$ converges in probability to $\mu$, it is enough to show that for every (bounded and) continuous function $f: [0,1]^2 \to \mathbb{R}$,
\begin{equation}\label{eq:fddconv2}
	P_n(f)\coloneqq\frac{1}{n}\sum_{i=1}^n f\left(\frac{i}{n},\frac{\sigma_n(i)}{n}\right)\xlongrightarrow[n\to\infty]{\mathbb{P}}\int_{[0,1]^2}f\,\,\mathrm{d}\mu=\mathbb{E}[f(X^1,Y^1)].
\end{equation}
But \eqref{eq:fddconv} immediately gives us that as $n$ tends to infinity,
\begin{equation}
	\mathbb{E}[P_n(f)]\to\mathbb{E}[f(X^1,Y^1)]
	\quad\text{and}\quad
	\mathbb{E}[P_n(f)^2]\to\mathbb{E}[f(X^1,Y^1)]^2
\end{equation}
and so a standard second moment argument shows that  \eqref{eq:fddconv} implies \eqref{eq:fddconv2}.

\begin{proof}[Proof of \Cref{thm:permuton_conv}]
	Let $E_{n,1},\dots,E_{n,n}$ be independent exponential r.v.'s with $\mathbb{E}[E_{n,i}]=1/\theta_{i}$. By the exponential representation of the Luce model (see for instance~\cite[Theorem 2.1]{chatterjee2023enumerative} and recall Remark~\ref{rmk:inv}), we can express $\sigma_n(i)$ as
	\[
	\sigma_n(i)=\sum_{j=1}^n\mathbf{1}\{E_{n,j}\le E_{n,i}\}.
	\]
	Note that for $j\neq i$,
	\[
	\mathbb{P}(E_{n,j}\le E_{n,i}\mid E_{n,i})=1-\mathrm{e}^{-\theta_{j}E_{n,i}}.
	\]
	Consequently, conditioning on $E_{n,i}$, $\sigma_n(i)$ is 1 plus a sum over $j\in \llbracket n \rrbracket\setminus\{i\}$ of independent Bernoulli random variables with
	means $1-\mathrm{e}^{-\theta_{j}E_{n,i}}$. Thus, if we let
	\[
	Z_{n,i}:=\frac{1}{n}+\frac{1}{n}\sum_{j\neq i}(1-\mathrm{e}^{-\theta_{j}E_{n,i}}),
	\]
	an application of Hoeffding's inequality gives us that for all $t\geq 0$,
	\[
	\mathbb{P}\left(\left|\frac{\sigma_n(i)}{n}-Z_{n,i}\right|\geq t \,\,\middle|\,\, E_{n,i} \right)\leq 2\mathrm{e}^{-nt^2/2}.
	\]
	Since the bound does not depend on $E_{n,i}$, the unconditional probability is also bounded by the
	same quantity. Taking a union bound over $i\in\llbracket n \rrbracket$ and applying the Borel–Cantelli lemma, this
	shows that almost surely as $n\to \infty$,
	\[
	\max_{1\le i\le n}\left|\frac{\sigma_n(i)}{n}-Z_{n,i}\right|\to 0.
	\]
	
	Now let $W^\ind_n=Z_{n,U^\ind_n}$. By the above display,  the vector $(Y^\ind_n-W^\ind_n)_{\ind\in \llbracket 2 \rrbracket}$ converges to the zero vector almost surely as $n\to \infty$.
	Thus, it suffices to study the convergence of the vector
	$(X^\ind_n,W^\ind_n)_{\ind\in\llbracket 2 \rrbracket}$. Moreover, from now on, we can also assume that $U^1_n\neq U^2_n$. Indeed, the probability of the complement event tends to zero as $n\to \infty$ and so the limit (if it exists) of the vector
	$(X^\ind_n,W^\ind_n)_{\ind\in\llbracket 2 \rrbracket}$ under the conditional law that $U^1_n\neq U^2_n$ is the same as the one obtained under the unconditional law.
	
	Note that, conditioning on  $U^\ind_n=i^{\ind}$ for $\ind\in\llbracket 2 \rrbracket$ with $i^1\neq i^2$, the law of $(W_n^\ind)_{\ind\in\llbracket 2 \rrbracket}$ is the same as that of
	\[
	(Q^{\ind}_{n,i^{\ind}})_{\ind\in\llbracket 2 \rrbracket}=\left(\frac{1}{n}+\frac{1}{n}\sum_{j\neq i^{\ind}}(1-\mathrm{e}^{-\theta_{j}E^{i^{\ind}}/\theta_{i^{\ind}}})\right)_{\ind\in\llbracket 2 \rrbracket}
	\]
	where the $(E^{i^{\ind}})_{\ind\in\llbracket 2 \rrbracket}$ are independent exponential random variables with mean $1$, which are also independent of $(U^\ind_n)_{\ind\in\llbracket 2 \rrbracket}$. Since
	$X^{\ind}_n = U^\ind_n/n$, this shows that the joint law of $(X^\ind_n,W^\ind_n)_{\ind\in\llbracket 2 \rrbracket}$ is the same as the joint law of $(X^\ind_n,Q^{\ind}_{n,U_n^\ind})_{\ind\in\llbracket 2 \rrbracket}$. So, our task reduces to identifying the limiting distribution of $(X^\ind_n,Q^{\ind}_{n,U_n^\ind})_{\ind\in\llbracket 2 \rrbracket}$. Since the random variables $(E^{i^{\ind}})_{\ind\in\llbracket 2 \rrbracket}$ and $(U^\ind_n)_{\ind\in\llbracket 2 \rrbracket}$ are all jointly independent, and we have the assumption that $U^1_n\neq U^2_n$, we have that $(X^1_n,Q^{1}_{n,U_n^1})$ and $(X^2_n,Q^{2}_{n,U_n^2})$ are independent. Therefore, recalling that the limiting random variables $(X^1,Y^1)$ and $(X^2,Y^2)$ are also independent, in order to conclude, it is enough to show that $(X_n,Q_{n,U_n})$ converges in distribution to $(X,Y)$, where $(X_n,Q_{n,U_n})$ is distributed as $(X^1_n,Q^1_{n,U^1_n})$ and $(X,Y)$ is distributed as $(X^1,Y^1)$.
	Now, notice that since $f_n(x)=\theta_{\lfloor xn \rfloor}$,
	\[
	Q_{n,U_n}
	=
	\frac{1}{n}-\frac{1-\mathrm{e}^{-E}}{n}+\frac{1}{n}\sum_{j=1}^n(1-\mathrm{e}^{-\theta_{j}E^{i^{\ind}}/\theta_{U_n}})
	=
	\frac{\mathrm{e}^{-E}}{n}+\int_{0}^1(1-\mathrm{e}^{-f_n(x)E/\theta_{U_n}}) \,\,\mathrm{d}x,
	\]
	where $E$ is an exponential random variable with mean $1$, which is  independent of $U_n$.
	
	Consequently, for any bounded continuous function $g : [0, 1]^2 \to \mathbb{R}$,
	\begin{align*}
		\mathbb{E}[g(X_n,Q_{n,U_n})]
		&=
		\mathbb{E}\left[g\left(X_n,\frac{\mathrm{e}^{-E}}{n}+\int_{0}^1(1-\mathrm{e}^{-f_n(x)E/\theta_{U_n}}) \,\,\mathrm{d}x\right)\right]\\
		&=
		\mathbb{E}\left[\frac{1}{n}\sum_{i=1}^n g\left(\frac{i}{n},\frac{\mathrm{e}^{-E}}{n}+\int_{0}^1(1-\mathrm{e}^{-f_n(x)E/\theta_{i}}) \,\,\mathrm{d}x\right)\right]\\
		&=
		\mathbb{E}\left[\int_{0}^1 g\left(h_n(y),\frac{\mathrm{e}^{-E}}{n}+\int_{0}^1(1-\mathrm{e}^{-f_n(x)E/f_n(y)}) \,\,\mathrm{d}x\right)\,\,\mathrm{d}y\right]\\
		&=
		\int_{0}^\infty\int_{0}^1 \mathrm{e}^{-z}g\left(h_n(y),\frac{\mathrm{e}^{-z}}{n}+\int_{0}^1(1-\mathrm{e}^{-f_n(x)z/f_n(y)}) \,\,\mathrm{d}x\right)\,\,\mathrm{d}y\,\mathrm{d}z,
	\end{align*}
	where $h_n(y) := \frac{\lfloor yn \rfloor}{n}$.
	Now, given any $y$ such that $f_n(y) \to f(y) > 0$, and any $z > 0$, we have that for almost
	every $x \in [0, 1]$,
	\[
	\lim_{n\to\infty}(1-\mathrm{e}^{-f_n(x)z/f_n(y)})=1-\mathrm{e}^{-f(x)z/f(y)}.
	\]
	Thus, by the dominated convergence theorem, we have that for any $y$ and $z$ as above,
	\[
	\lim_{n\to\infty}\int_{0}^1(1-\mathrm{e}^{-f_n(x)z/f_n(y)})=\int_{0}^1(1-\mathrm{e}^{-f(x)z/f(y)})\,\,\mathrm{d}x.
	\]
	But again, $f_n(y) \to f(y) > 0$ for almost every $y \in [0, 1]$. Thus, by the boundedness and
	continuity of $g$, and the dominated convergence theorem,
	\begin{align*}
		\lim_{n\to\infty}\mathbb{E}[g(X_n,Q_{n,U_n})]
		&=
		\lim_{n\to\infty}\int_{0}^\infty\int_{0}^1 \mathrm{e}^{-z}g\left(h_n(y),\frac{\mathrm{e}^{-z}}{n}+\int_{0}^1(1-\mathrm{e}^{-f_n(x)z/f_n(y)}) \,\,\mathrm{d}x\right)\,\,\mathrm{d}y\,\mathrm{d}z\\
		&=
		\int_{0}^\infty\int_{0}^1 \mathrm{e}^{-z}g\left(y,\int_{0}^1(1-\mathrm{e}^{-f(x)z/f(y)}) \,\,\mathrm{d}x\right)\,\,\mathrm{d}y\,\mathrm{d}z.
	\end{align*}
	The last expression equals $\mathbb{E}[g(X, Y )]$, where $X$ and $Y$ are as in the statement of the theorem.
	This completes the proof.
\end{proof}

\subsection{Properties of the Luce permuton}\label{Sect:prop-luce}

We now investigate some properties of the permuton $\mu$ in \Cref{thm:permuton_conv}.

Recall that given a permuton $\mu$, one can sample $n$ independent points $Z_1, \dots , Z_n$ in the unit
square $[0, 1]^2$ according to $\mu$. These $n$ points induce a random permutation $\sigma$: for any $i,j\in\llbracket n \rrbracket:=\{1,\dots,n\}$, let $\sigma(i) = j$ if the point with $i$-th lowest $x$-coordinate has $j$-th lowest $y$-coordinate (this is well-defined since the marginals of a permutons are uniform, and so, almost surely there are no points with the same $x$- or $y$-coordinates). We denote this permutation by $\Perm(\mu,n)$ and call it the \textbf{random permutation induced by the permuton $\mu$} of size $n$.

\begin{prop}\label{prop:density}
	The permuton $\mu$ in \Cref{thm:permuton_conv} is uniquely determined by its pattern densities, defined for all $k\geq 1$ and $\pi\in\mathcal{S}_k$ by
	\begin{equation}\label{eq:patt_dens}
		\pocc(\pi,\mu):=\BB P\big(\Perm(\mu,k)=\pi\big)= \BB E\Big[\BB P\left(\Luce\big(f(U^1),\dots,f(U^k)\big)=\pi \,\,\Big| \,\,U^1,\dots,U^k\right)  \Big],
	\end{equation}
	where $(U^1,\dots,U^k)$ are the order statistics of $k$ i.i.d.\ uniform random variables in $[0,1]$.
	
	Moreover, the permuton $\mu$ is absolutely continuous w.r.t.\ the uniform permuton $\Leb_{[0,1]^2}$ and it has the following density:
	\begin{equation}\label{eq:density}
		\rho(x,y):=	\frac{f(x)\mathrm{e}^{-f(x)F^{-1}(y)}}{\int_0^1f(t)\mathrm{e}^{-f(t)F^{-1}(y)}\,\dd t}, \quad\text{for all $(x,y)\in(0,1)^2$},
	\end{equation}
	where $F$ is as in \eqref{eq:defn-F}.
\end{prop}

Before proving the proposition above, we provide some remarks and examples.

\medskip

We start by noting that the first part of the proposition statement heuristically says that the random permutation induced by the Luce permuton $\mu$ of size $k$ is distributed as the average of an exact Luce-distributed permutation with \emph{random} weights $f(U^1),\dots,f(U^k)$.

\medskip

If $f(x)=1-x$, i.e.\ $\mu$ is the permuton limit of Luce-distributed permutations with Sukhatme weights $\theta_i=n-i+1$, then $F(x)=\frac{x-1 + \mathrm{e}^{-x} }{x}$. Hence, denoting by $\varphi(x)$ the inverse of $\frac{\mathrm{e}^{-x}+x-1}{x}$ (which does not have an explicit closed form but can be easily computed numerically), we get that the density of $\mu$ is given by
\begin{equation}\label{eq:suk-dens}
	\rho(x,y):=\frac{(1-x)\mathrm{e}^{-(1-x)\varphi(y)}}{\int_0^1(1-t)\mathrm{e}^{-(1-t)\varphi(y)}\,\dd t}
	=
	\frac{(1-x)\varphi(y)^2\mathrm{e}^{x\varphi(y)}}{\mathrm{e}^{\varphi(y)}-\varphi(y)-1}, \quad\text{for all $(x,y)\in(0,1)^2$}.
\end{equation}
The graph of the density $\rho(x,y)$ is plotted in the second panel of \Cref{fig:mallows_luce}. 

We highlight two important features of the Luce permuton $\mu$ with Sukhatme weights: 
\begin{itemize}
	\item Its density $\rho(x,y)$ is singular at $(x,y)=(1,1)$. Indeed, noting that $\frac{x-1 + \mathrm{e}^{-x} }{x}= 1-1/x+o(1/x)$, we get $\varphi(x)\sim_{x\to 1} 1/(1-x)$ and so $\lim_{x\to 1^{-}}\rho(x,x)=+\infty$. We note that this singular behavior is different from the behavior of other well-studied limiting permutons: the Mallows permutons $(\mu_{\beta})_{\beta>0}$, i.e.\ the permuton limits of Mallows distributed permutations\footnote{We recall that in the Mallows model on permutations, the probability of a permutation is proportional to a real parameter $q$ raised to the power of the number of inversions of the permutation.} with parameter $q=1-\beta/n$ \cite{starr2009thermodynamic}. Indeed, it is simple to see that the densities 
	\begin{equation}
		\rho_{\beta}(x,y)=\frac{\beta/2 \sinh[\beta/2]}{\left(\mathrm{e}^{\beta/4} \cosh[\beta/2 (x - y)] - 
			\mathrm{e}^{-\beta/4} \cosh[\beta/2 (x + y - 1)]\right)^2},
	\end{equation}
	of the Mallows permutons $(\mu_{\beta})_{\beta>0}$ are bounded for all $\beta>0$.
	
	\item The density $\rho_{\beta}(x,y)$ of the Mallows permuton is symmetric with respect to the main diagonal of the unit square. In contrast, the density $\rho(x,y)$ of the Luce permuton with Sukhatme weights is asymmetric.
\end{itemize}

In \Cref{fig:mallows_luce}, we compared the density of the permuton $\mu$ corresponding to Luce-distributed permutations with Sukhatme weights and the Mallows permuton $\mu_\beta$ with $\beta=6$.

\begin{figure}[h!]
	\centering
	\includegraphics[width=0.24\textwidth]{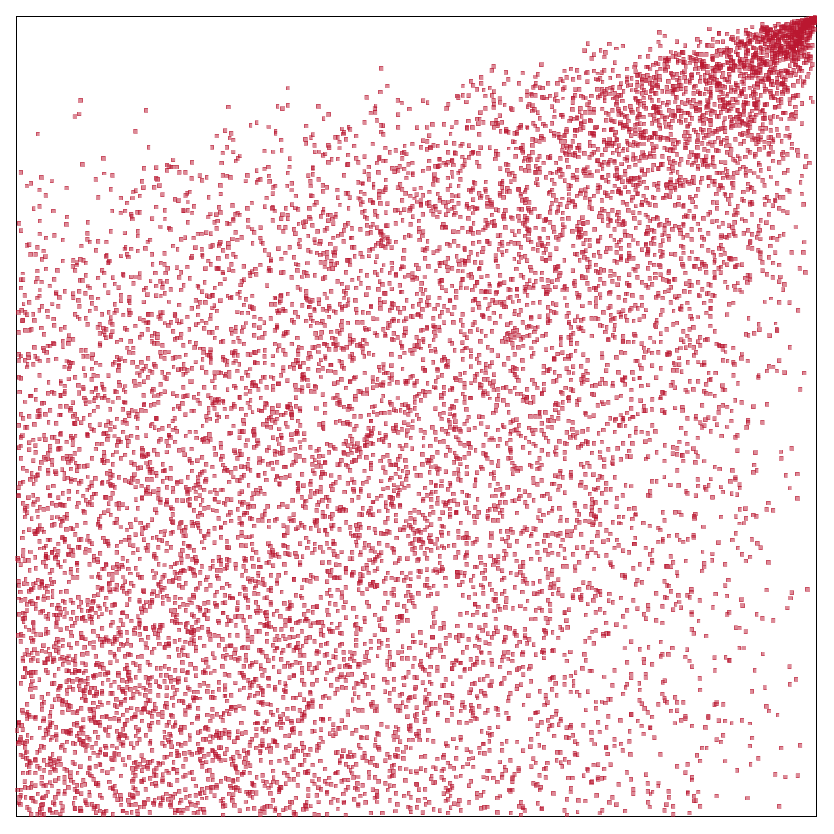}
	\includegraphics[width=0.24\textwidth]{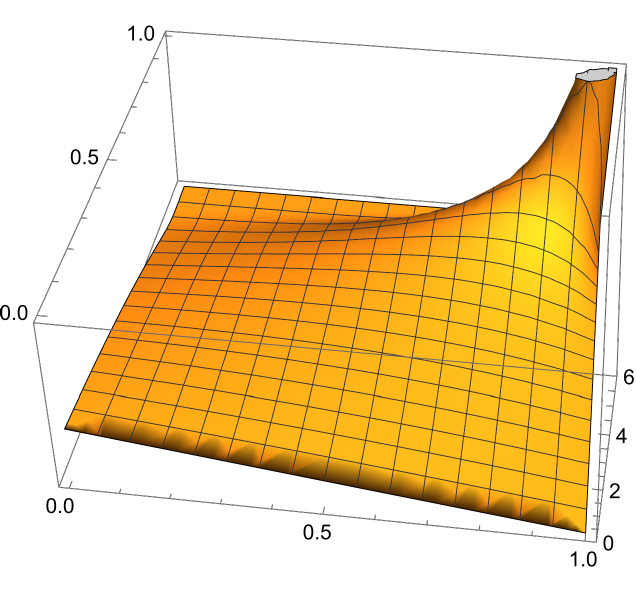}
	\includegraphics[width=0.24\textwidth]{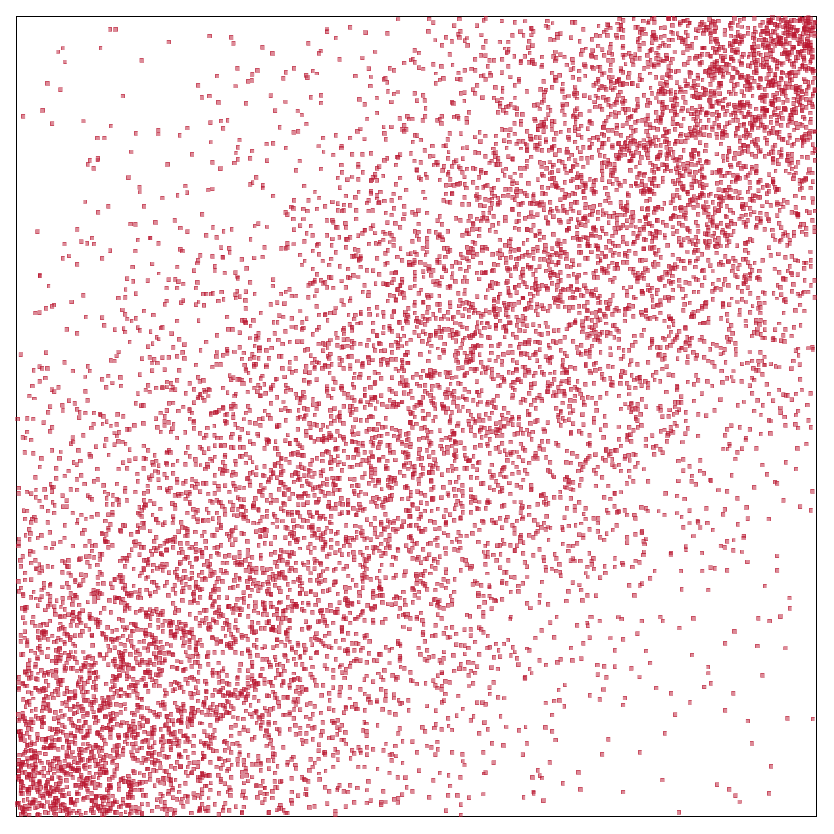}
	\includegraphics[width=0.24\textwidth]{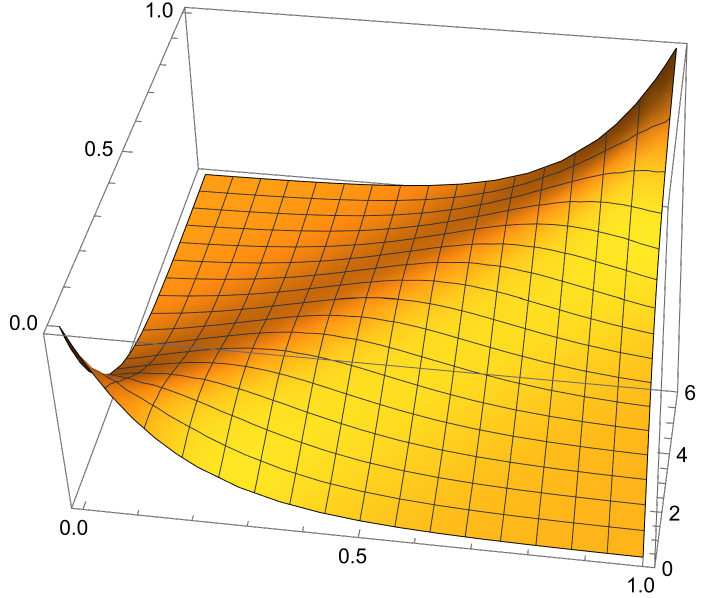}
	\caption{\textbf{From left to right:} (1) The diagram of a Luce-distributed permutation with Sukhatme weights of size $10000$; (2) The density $\rho(x,y)$ of the permuton limit of Luce-distributed permutation with Sukhatme weights; (3) The diagram of a Mallows distributed permutation with parameter $q=(1-6/n)$ of size $n=10000$; (4) The density $\rho_6(x,y)$ of the permuton limit $\mu_6$ of Mallows distributed permutation with parameter $q=(1-6/n)$.}
	\label{fig:mallows_luce}
\end{figure}

We now turn to the proof of Proposition~\ref{prop:density}.

\begin{proof}[Proof of Proposition~\ref{prop:density}]
	We first prove the second claim in the proposition statement.
	Recalling from \Cref{thm:permuton_conv} that $F(x)=1-\int_0^1\mathrm{e}^{-x f(y)}dy$, we get that for all $x>0$,
	\[
	F'(x)=\int_0^1f(y)\mathrm{e}^{-x f(y)}\,\dd y>0.
	\] 
	Therefore, $F$ is a (strictly) increasing diffeomorphism from the positive real line to $(0,1)$. In particular, $F$ is invertible.
	
	Let $g$ be a non-negative measurable function defined on the unit square $[0,1]^2$. We have that
	\begin{equation*}
		\BB E \left[g(U,F(E/f(U)))\right]
		=
		\int_0^1 \int_0^\infty g(x,F(s/f(x)))\mathrm{e}^{-s}\,\dd s\,\dd x.
	\end{equation*}
	Using the change of variables $y=F(s/f(x))$ and $s=f(x)F^{-1}(y)$, and noting that
	\begin{equation*}
		\dd s= f(x) (F^{-1})'(y)\, \dd y=\frac{f(x)}{\int_0^1f(t)\mathrm{e}^{-f(t)F^{-1}(y)}\,\dd t}\,\dd y,
	\end{equation*}
	we conclude that
	\begin{equation*}
		\BB E \left[g(U,F(E/f(U)))\right]
		=
		\int_0^1 \int_0^1 g(x,y)\mathrm{e}^{-f(x)F^{-1}(y)}
		\frac{f(x)}{\int_0^1f(t)\mathrm{e}^{-f(t)F^{-1}(y)}\,\dd t}\,\dd y\,\dd x,
	\end{equation*}
	finishing the proof of the second part of the proposition.
	
	\medskip
	
	We now prove the first part of the proposition. Fix $k\geq 1$ and $\pi\in\mathcal{S}_k$. We first compute $\pocc(\pi,\mu):=\BB P\big(\Perm(\mu,k)=\pi\big)$. By definition of $\mu$ and $\Perm(\mu,k)$ we have that
	\begin{multline*}
		\BB P\big(\Perm(\mu,k)=\pi\big)\\
		=\BB P\Big(F\big(E_1/f(U^{\pi^{-1}(1)})\big)<F\big(E_2/f(U^{\pi^{-1}(2)})\big)<\dots<F\big(E_k/f(U^{\pi^{-1}(k)})\big)\Big),
	\end{multline*}
	where $(E_i)_i$ are independent exponential random variables of parameter 1, and  $(U^1,\dots,U^k)$ are the order statistics of $k$ i.i.d.\ uniform random variables in $[0,1]$ also independent of $(E_i)_i$.
	Since $F$ is a (strictly) increasing diffeomorphism (as shown above), the previous expression rewrites as
	\[
	\BB P\big(\Perm(\mu,k)=\pi\big)=\BB P\Big(E_1/f(U^{\pi^{-1}(1)})<E_2/f(U^{\pi^{-1}(2)})<\dots<E_k/f(U^{\pi^{-1}(k)})\Big).
	\]
	Conditioning on $(U^1,\dots,U^k)$, the random variables $\left(E_i/f(U^{\pi^{-1}(i)})\right)_i$ are independent exponential random variables of parameters $\left(f(U^{\pi^{-1}(i)})\right)_i$, respectively. Hence, \eqref{eq:patt_dens} immediately follows from \cite[Theorem 2.1]{chatterjee2023enumerative}.
	
	Finally, the fact that the pattern densities $(\pocc(\pi,\mu))_{k\geq 1,\pi\in\mathcal{S}_k}$ uniquely determine the permuton $\mu$ is a consequence of the fact that $\Perm(\mu,k)$ converges in probability to $\mu$ in the permuton sense (see, for instance, \cite[Lemma 4.2]{hoppen2013limits}). We conclude by the uniqueness of the limit.
\end{proof}

\subsection{Convergence for the proportion of pattern occurrences}\label{sect:patt-dens}

\bigskip

Thanks to \cite[Theorem 2.5]{bbfgp-universal}, an immediate consequence of the permuton convergence in \Cref{thm:permuton_conv} is the following (joint) law of large number for the number of occurrences of any pattern in Luce-distributed permutations.

\begin{cor}
	Assume that $\sigma_n\sim \Luce(\theta_1,\dots,\theta_n)$ and that it satisfies the assumption of \Cref{thm:permuton_conv}.
	The random infinite vector $\big(\pocc(\pi,\sigma_n)\big)_{k\geq 1,\pi \in \mcl S_k}$ converges in probability w.r.t.\ the product topology to the infinite vector $(\pocc(\pi,\mu))_{k\geq 1,\pi \in \mcl S_k}$, where $\pocc(\pi,\mu)$ was introduced in \eqref{eq:patt_dens}.
\end{cor}

In the next table, we explicitly compute the values of $\pocc(\pi,\mu)$ for patterns of size $k=2,3$, in the case of $f(x)=1-x$ (i.e.\ when $\mu$ is the permuton limit of Luce-distributed permutations with Sukhatme weights $\theta_i=n-i+1$) using the formula
\begin{multline*}
	\pocc(\pi,\mu)= \BB E\Big[\BB P\left(\Luce\big(f(U^1),\dots,f(U^k)\big)=\pi \Big| U^1,\dots,U^k\right)  \Big]\\
	=k!\int_{0<u_1<\dots<u_k<1}p\left(1-u_{\pi^{-1}(1)},\dots,1-u_{\pi^{-1}(k)}\right) \,\dd u_1 \dots \dd u_k,
\end{multline*}
where we recall that the function $p$ was introduced in \eqref{eq:p-function}. As seen, the chances are far from uniform. This is in contrast with the case of local patterns: As we will show in~\Cref{the:unif-luce}, local patterns for Luce-distributed permutations with Sukhatme weights $\theta_i=n-i+1$ behave uniformly.


\medskip

\begin{center}
	\begin{tabular}{ |c|c|c| }
		\hline
		\multicolumn{3}{|c|}{Values of $\pocc(\pi,\mu)$ for the Luce permuton with Sukhatme weights} \\
		\hline
		$\pi$ & Exact values for $\pocc(\pi,\mu)$ & Numerical values for $\pocc(\pi,\mu)$\\
		\hline
		12   & $\log(2)$    & 0.69315 \\
		21&   $1-\log(2)$  & 0.30685   \\
		123 & $\frac 1 4 \left(2-\log(27/16)\right)$ & 0.36919\\
		213& $\log(256/27)-2$  & 0.24934\\
		132    & $6\left(\frac{5\log(3)}{8}-\frac{1}{12}-\frac{5\log(2)}{6}\right)$ & 0.15406\\
		231&   $2-\log(27/4)$  & 0.09046\\
		312& $6\left(\frac{1}{4}-\frac{4\log(2)}{3}-\frac{5\log(3)}{8}\right)$  & 0.07462\\
		321& $\frac 1 4 \left(\log(256/27)-2\right)$& 0.06234\\
		\hline
	\end{tabular}
\end{center}

\subsection{Differences between exact Luce-distributed permutations and permutations sampled from the corresponding Luce permuton}\label{sect:diff-sim}

The permuton limit of a sequence of permutations encodes well the ``global properties'' of the sequence of permutations, such as the proportion of patterns, as we saw in the previous section. On the other hand, the permuton limit (a priori) does not encode finer properties of the sequence of permutations, such as the distribution of the first or the last values of the permutations.

Hence, a natural question is to study how a sequence of exactly distributed Luce permutations behaves differently than the permutations sampled for their limiting permuton. More precisely, here we investigate the following question: in \cite[Theorem 4.2]{chatterjee2023enumerative}, the limiting distribution for the position of the $k$ smallest values of a Luce distributed permutation was determined.\footnote{With their definition of Luce distributed permutations (recall Remark~\ref{rmk:inv}) this is the limiting distribution for the top $k$ cards.} More precisely, it was proved that for any $k\geq 1$ and any distinct positive integers $a_1,\dots,a_k$, (recall Remark~\ref{rmk:inv})

\begin{multline}\label{eq:last_val}
	\BB P (\sigma_n(a_1)=1,\dots, \sigma_n(a_k)=k)\\
	=\int_{x_1>x_2>\dots>x_k>0}\prod_{j=1}^k\left(\theta_{a_j}\mathrm{e}^{-\theta_{a_j}x_j}\right)\prod_{i\in\llbracket n \rrbracket\setminus\{a_1,\dots,a_k\}}\left(1-\mathrm{e}^{-\theta_i x_k}\right) \, \dd x_1\dots\dd x_k,
\end{multline}
so that $\lim_{n\to\infty} \BB P (\sigma_n(a_1)=1,\dots, \sigma_n(a_k)=k)$ is equal to
\begin{equation*}
	\int_{x_1>x_2>\dots>x_k>0}\prod_{j=1}^k\left(\theta_{a_j}\mathrm{e}^{-\theta_{a_j}x_j}\right)\prod_{i\notin\{a_1,\dots,a_k\}}\left(1-\mathrm{e}^{-\theta_i x_k}\right) \, \dd x_1\dots\dd x_k.
\end{equation*}
What about the analog question when $\sigma_n$ is replaced by $\pi_n=\Perm(\mu,n)$, where $\mu$ is the permuton limit of $\sigma_n$ (if it exists)?

\begin{prop}\label{prop:last_val_perm}
	Let $\mu$ be the Luce permuton corresponding to a positive, finite, measurable function $f$ on $[0,1]$, as defined in \Cref{thm:permuton_conv}. Set $\pi_n=\Perm(\mu,n)$. Then
	\begin{multline*}
		\BB P \left(\pi_n(a_1)=1,\dots, \pi_n(a_k)=k\right)\\
		=\BB E\left[\int_{x_1>x_2>\dots>x_k>0}\prod_{j=1}^k\left(f(U^{a_j})\mathrm{e}^{-f(U^{a_j})x_j}\right)\prod_{i\in\llbracket n \rrbracket\setminus\{a_1,\dots,a_k\}}\left(1-\mathrm{e}^{-f(U^{i}) x_k}\right) \, \dd x_1\dots\dd x_k\right],
	\end{multline*}
	where $(U^1,\dots,U^n)$ are the order statistics of $n$ i.i.d.\ uniform random variables in $[0,1]$.
\end{prop}

\begin{proof}
	This is a simple consequence of the formula in \eqref{eq:last_val} and the characterization of  $\Perm(\mu,n)$ given in Proposition~\ref{prop:density}.
\end{proof}

If $\sigma_n\sim \Luce(\theta_1,\dots,\theta_n)$ and $\mu$ is the corresponding Luce permuton limit, then one could naively guess that since $f=\lim_{n\to\infty}\theta_{\lfloor yn \rfloor}/w_n$ for a.e.\ $y\in[0,1)$ then
\begin{multline}\label{eq:false_approx}
	\BB E\left[\int_{x_1>x_2>\dots>x_k>0}\prod_{j=1}^k\left(f(U^{a_j})\mathrm{e}^{-f(U^{a_j})x_j}\right)\prod_{i\in\llbracket n \rrbracket\setminus\{a_1,\dots,a_k\}}\left(1-\mathrm{e}^{-f(U^{i}) x_k}\right) \, \dd x_1\dots\dd x_k\right]\\
	=
	\BB E\left[\int_{x_1>x_2>\dots>x_k>0}\prod_{j=1}^k\left(\frac{f(U^{a_j})}{f(U^{1})}\mathrm{e}^{-\frac{f(U^{a_j})}{f(U^{1})}x_j}\right)\prod_{i\in\llbracket n \rrbracket\setminus\{a_1,\dots,a_k\}}\left(1-\mathrm{e}^{-\frac{f(U^{i})}{f(U^{1})} x_k}\right) \, \dd x_1\dots\dd x_k\right]
	\approx\\
	\int_{x_1>x_2>\dots>x_k>0}\prod_{j=1}^k\left(\theta_{a_j}\mathrm{e}^{-\theta_{a_j}x_j}\right)\prod_{i\notin\{a_1,\dots,a_k\}}\left(1-\mathrm{e}^{-\theta_i x_k}\right) \, \dd x_1\dots\dd x_k.
\end{multline}
This turns out to be false! Indeed, if for instance $\theta_i=i$, so that $f(x)=x$, then 
\begin{equation*}
	\frac{f(U^{a_j})}{f(U^{1})}=\frac{U^{a_j}}{U^{1}}\stackrel{d}{=}\frac{E_1+\dots+E_{a_j}}{E_1},
\end{equation*}
where we used that $(U^1, \dots, U^n)$ is equal in distribution to 
\[\left(\frac{E_1}{\sum_{i=1}^{n+1} E_i}, \frac{E_1+E_2}{\sum_{i=1}^{n+1} E_i}, \dots, \frac{E_1+\dots+E_n}{\sum_{i=1}^{n+1} E_i}\right),\] 
with $(E_i)_i$ i.i.d.\ exponential random variables of parameter 1. And so, $\frac{f(U^{a_j})}{f(U^{1})}$ does not concentrate around $a_j$.

\medskip

Given this observation, we wondered how different the expressions on the left-hand side and right-hand side of \eqref{eq:false_approx} are. We looked, for instance, at the case of $\BB P(\sigma_n(1)=1)$ and $\BB P(\Perm(\mu,n)(1)=1)$ when $\theta_i=i$ and $f(x)=x$. The results are shown in the next table.

\medskip

\begin{center}
	\begin{tabular}{ |c|c|c| }
		\hline
		$n$ & $\BB P(\Perm(\mu,n)(1)=1)$ & $\BB P(\sigma_n(1)=1)$\\
		\hline
		10   &  0.4641 & 0.5184 \\
		20   & 0.5049  & 0.5162 \\
		50   & 0.5339 &  0.5161 \\
		100   &  0.5443 & 0.5161 \\
		1000   & 0.5540  & 0.5161 \\
		10000   & 0.5551  & 0.5161 \\
		\hline
	\end{tabular}
\end{center}

We note that the results are different but roughly similar.

\section{Local limits and a CLT for consecutive occurrences}\label{sect:local}

In the previous sections, we looked at the permuton limit for Luce-distributed permutations. Here, we focus on the natural counterpart of local limits, following the framework introduced in \cite{borga2020local}. These are extensions of the classical descents.

\subsection{Statement of the main results}

For the Luce model, the quenched Benjamini--Schramm convergence follows from the following result, whose proof is postponed to \Cref{sect:proofs-local}.

\begin{thm}[Local limit]\label{thm:lln_local}
	Fix $k\geq 1$ and $\pi\in\mcl S_k$. Assume that $\sigma_n\sim \Luce(\theta_1,\dots,\theta_n)$ and that
	\begin{equation}\label{eq:ass-clt-cons}
		\Lambda(\pi):=\lim_{n\to +\infty}\frac{1}{n}\sum_{i=0}^{n-k}p\left(\theta_{i+\pi^{-1}(1)},\dots, \theta_{i+\pi^{-1}(k)}\right)
	\end{equation}
	exists. 
	Then, the following convergence in probability holds
	\begin{equation*}
		\pcocc(\pi,\sigma_n)\to\Lambda(\pi).
	\end{equation*}
\end{thm}

The assumption in \eqref{eq:ass-clt-cons} is a regularity assumption for the weights $\theta_1,\dots,\theta_n$ that will be verified in \Cref{thm:general-local-thm} below under a minimal assumption. For instance, if $\pi=21$, i.e.\ the case of descents, then \eqref{eq:ass-clt-cons} requires the existence of a finite limit for:
\begin{equation*}
	\frac{1}{n}\sum_{i=0}^{n-2}\frac{\theta_{i+2}}{\theta_{i+1}+\theta_{i+2}}.
\end{equation*}

We also have the following stronger result (whose proof is also postponed to the end of the section). It gives a central limit theorem for these local statistics with a Berry--Esseen type error estimate.

\begin{thm}\label{thm:clt_local}
	Fix $k\geq 1$ and $\pi\in\mcl S_k$. Assume that $\sigma_n\sim \Luce(\theta_1,\dots,\theta_n)$ and let
	\begin{align}\label{eq:defn-variance}
		(\nu_{n}(\pi))^2:=&\Var\left(\cocc(\pi,\sigma_n)\right)\notag\\
		=&\sum_{i=0}^{n-k}\left(p\left(\theta_{i+\pi^{-1}(1)},\dots, \theta_{i+\pi^{-1}(k)}\right)-p\left(\theta_{i+\pi^{-1}(1)},\dots, \theta_{i+\pi^{-1}(k)}\right)^2\right)\notag\\
		&+
		2\sum_{h=1}^{k-1}\sum_{i=0}^{n-k-h}\Bigg(  \BB P\left(E_{i+\pi^{-1}(1)}<\dots<E_{i+\pi^{-1}(k)}, E_{i+h+\pi^{-1}(1)}<\dots<E_{i+h+\pi^{-1}(k)}\right)\notag\\
		&-p\left(\theta_{i+\pi^{-1}(1)},\dots, \theta_{i+\pi^{-1}(k)}\right)p\left(\theta_{i+h+\pi^{-1}(1)},\dots, \theta_{i+h+\pi^{-1}(k)}\right)\Bigg),
	\end{align}
	where $(E_i)_i$ are independent exponential random variables of parameter $(\theta_{i})_i$ respectively.
	Then, the following holds
	\begin{equation*}
		\left|\BB P\left(\frac{\cocc(\pi,\sigma_n)-\BB E\left[\cocc(\pi,\sigma_n)\right]}{\nu_{n}(\pi)}\leq w\right)-\Phi(w)\right|\leq 32(1+\sqrt 6)\frac{nk^2}{(\nu_{n}(\pi))^{3}},
	\end{equation*}
	where $\Phi(w)$ denotes the cumulative distribution function of the standard normal distribution.
\end{thm}

Note that when $\pi=21$, i.e.\ the case of descents, then, setting $\widetilde{\theta}_i=\frac{\theta_{i+2}}{\theta_{i+1}+\theta_{i+2}}$, we have that
\begin{align*}
	(\nu_{n}(21))^2
	=\sum_{i=0}^{n-2}\left(\widetilde{\theta}_i-\left( \widetilde{\theta}_i \right)^2\right)
	+
	2\sum_{i=0}^{n-3}\Bigg(  \BB P\left(E_{i+3}<E_{i+2}<E_{i+1}\right)
	-\widetilde{\theta}_i \cdot \widetilde{\theta}_{i+1}\Bigg).
\end{align*}


It is natural to wonder when the Assumption~\eqref{eq:ass-clt-cons} in \Cref{thm:lln_local} holds. The next theorem shows that such assumption holds\footnote{See \Cref{sect:app2} for another setting where this assumption still holds.} as soon as the assumption of \Cref{thm:permuton_conv} is satisfied, i.e.\ the one in our permuton limit result. It also explicitly compute the asymptotics for $(\nu_{n}(\pi))^2$ from \Cref{thm:clt_local}.

\newcommand{\R}{\mathbb{R}}
\newcommand{\1}{\mathbf{1}}

\begin{thm}\label{thm:general-local-thm}
	Set $f_n(y):=\theta_{\lfloor yn \rfloor}$ for all $y\in[0,1]$ and assume that $f_n\to f$ almost everywhere for some positive, finite, measurable function $f$ on $[0,1]$. 
	Then, for all $k\geq 1$ and every permutation $\pi\in\mcl{S}_k$,
	\begin{equation}\label{eq:lim-exp-loc}
		\Lambda(\pi)=\lim_{n\to +\infty}\frac{1}{n}\sum_{i=0}^{n-k}p\left(\theta_{i+\pi^{-1}(1)},\dots, \theta_{i+\pi^{-1}(k)}\right)=\frac{1}{k!},
	\end{equation}
	and 
	\begin{equation}\label{eq:lim-var-loc}
		\lim_{n\to\infty}\frac{(\nu_n(\pi))^2}{n}=\nu_\infty(\pi),
	\end{equation}
	where 
	\begin{equation}\label{eq:nu-pi}
		\nu_\infty(\pi)
		=\frac{1}{k!}+\frac{1-2k}{(k!)^2}+2 \sum_{h=1}^{k-1}\zeta_{\pi}(h,k),
	\end{equation}
	with
	\begin{equation}\label{eq:wefvbiwebfow}
		\zeta_{\pi}(h,k):=\mathbb P\!\left(
		E_{\pi^{-1}(1)}<\dots<E_{\pi^{-1}(k)},
		\;\; E_{h+\pi^{-1}(1)}<\dots<E_{h+\pi^{-1}(k)}
		\right),\quad\text{$1\le h\le k-1$,}
	\end{equation}
	where $(E_{j})_{j}$ are i.i.d.\ exponential random variables with parameter 1.
\end{thm}

Consequently, thanks to \Cref{thm:lln_local} and \Cref{thm:clt_local}, under the assumption of \Cref{thm:general-local-thm} on the weights $\theta_1,\dots,\theta_n$, if $\sigma_n\sim \Luce(\theta_1,\dots,\theta_n)$, then it behaves locally as a uniform permutation, in the sense of quenched Benjamini–Schramm convergence, Moreover,
\begin{equation}\label{eq:clt-unif-luce}
	\sqrt{n} \,\frac{\pcocc(\pi,\sigma_n)-\frac{1}{k!}}{\sqrt{\nu_\infty(\pi)}}\xrightarrow{\mathrm{d}}\mcl{N}(0,1),
\end{equation}
where $\nu_\infty(\pi)$ is defined as in \eqref{eq:nu-pi}.

Note that these are the same values one obtains for uniform permutations (see \Cref{sect:app1} for even more explicit computations). \Cref{eq:clt-unif-luce} is reminiscent of results showing that local statistics -- such as descents -- have limiting distributions that do not change between uniform and non-uniform models. For example, Kammoun~\cite{kammoun2022universality} proves such universality for the Mallows models, and Kim and Li~\cite{kim2020central} show that descents are asymptotically normal when restricted to permutations in a fixed conjugacy class. See also~\cite{borga2021asymptotic, guionnet2023universality, feray2024asymptotic} for other similar universality results.

One could also investigate the finer properties of local patterns. For instance, it would be nice to study whether, in the above setting, the descents form  the same determinantal point process as those for descents in uniform permutations; see~\cite[Theorem 5.1]{borodin2010adding} for more details.

\medskip
Before proving \Cref{thm:lln_local}, \Cref{thm:clt_local}, and \Cref{thm:general-local-thm}, we provide some applications, including explicit computations for the quantities appearing in these three theorems.

\subsection{Application 1: permutations with Sukhatme weights are locally uniform}\label{sect:app1}

We consider the canonical example of Sukhatme weights $\theta_i=n-i+1$. In this case, $f_n(x) \to f(x) = 1 - x$ for all $x \in [0, 1]$. Since $f(x) > 0$ for a.e.\ $x$, we can apply \Cref{thm:general-local-thm} and obtain the following result.

\begin{prop}\label{the:unif-luce}
	Let $\sigma_n$ be a Luce-distributed permutation of size $n$ with Sukhatme weights $\theta_i=n-i+1$. Then
	\[\big(\pcocc(\pi,\sigma_n)\big)_{k \geq 1,\pi\in\mathcal{S}_{k}}\xrightarrow{P}\left(\frac{1}{k!}\right)_{k \geq 1,\pi\in\mathcal{S}_{k}}\] 
	w.r.t.\ the product topology. In particular, $\sigma_n$ quenched Benjamini--Schramm converges. 
	Moreover, for every fixed $k\geq 1$ and $\pi\in\mcl S_k$,
	\begin{equation}
		\sqrt{n} \,\frac{\pcocc(\pi,\sigma_n)-\frac{1}{k!}}{\sqrt{\nu_\infty(\pi)}}\xrightarrow{\mathrm{d}}\mcl{N}(0,1),
	\end{equation}
	where $\nu_\infty(\pi)$ is defined as in \eqref{eq:nu-pi}.
\end{prop}

We show the explicit values for $\nu_\infty(\pi)$ for all patterns of size 2 and 3 (with some matching numerical simulations shown in \Cref{fig:CLT_local_luce}):
\begin{align*}
	&\nu_\infty(12)=\nu_\infty(21)=\frac{1}{12},\\  
	&\nu_\infty(123)=\nu_\infty(321)=\frac{23}{180} , \qquad \nu_\infty(132)=\nu_\infty(312)=\nu_\infty(213)=\nu_\infty(231)=\frac{7}{90}.
\end{align*}

\begin{figure}[h!]
	\centering
	\includegraphics[width=0.32\textwidth]{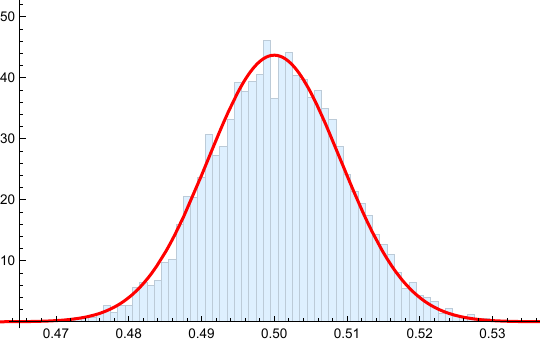}
	\includegraphics[width=0.32\textwidth]{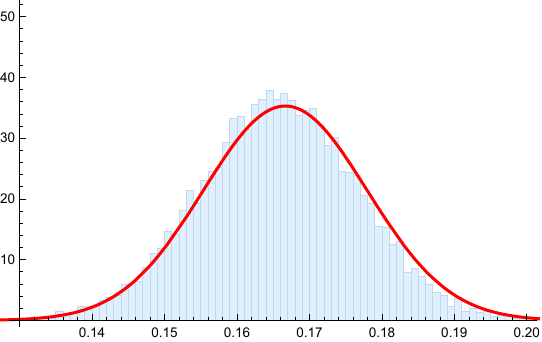}
	\includegraphics[width=0.32\textwidth]{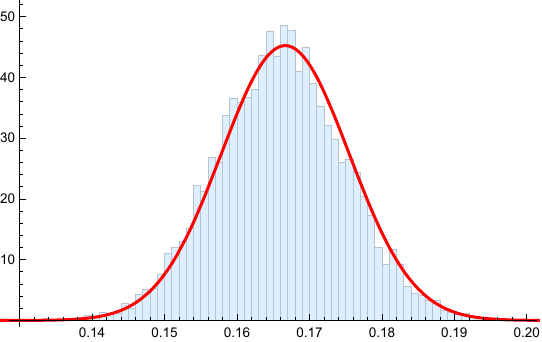}
	\caption{Simulations for the proportion of  consecutive patterns $\pcocc(\pi,\sigma_n)$ when $\sigma_n$ is a Luce-distributed permutation of size $n$ with Sukhatme weights $\theta_i=n-i+1$. In blue we show the histogram (renormalized to be a probability distribution) of the data collected from $3000$ random samples of size $1000$. In red we plot the density of $\mcl{N}(\frac{1}{k!},\nu_\infty(\pi)/1000)$. \textbf{From left to right:} (1) The case of descents, i.e.\ $\pi=21$. (2) The case of $\pi=321$ (3) The case of $\pi=231$.}
	\label{fig:CLT_local_luce}
\end{figure}

\subsection{Application 2: permutations with exponential Sukhatme weights are locally Luce-distributed}\label{sect:app2}

Here, we consider the example of exponential Sukhatme weights $\theta_i=\alpha^{n-i+1}$ for some $\alpha>0$. Note that the assumption of \Cref{thm:general-local-thm} is \emph{not} satisfied in this setting, but we will show that we can still prove local convergence for this model and a central limit theorem for consecutive patterns.

\begin{lem}\label{lem:tecn_estm2}
	Fix $k\geq 1$ and $\pi\in\mcl S_k$. Consider the case of exponential Sukhatme weights $\theta_i=\alpha^{n-i+1}$ for some $\alpha>0$. Then \[\Lambda(\pi)=\lim_{n\to +\infty}\frac{1}{n}\sum_{i=0}^{n-k}p\left(\theta_{i+\pi^{-1}(1)},\dots, \theta_{i+\pi^{-1}(k)}\right)=p\left(\alpha^{-\pi^{-1}(1)},\dots, \alpha^{-\pi^{-1}(k)}\right).\]
	Moreover, setting for all $k\geq 1$ and $1\leq h\leq k-1$,
	\begin{equation*}
		\eta_{\pi}(h,k):=\BB P\left(E_{\pi^{-1}(1)}<\dots<E_{\pi^{-1}(k)}, \,\, E_{h+\pi^{-1}(1)}<\dots<E_{h+\pi^{-1}(k)}\right), 
	\end{equation*}
	where $(E_{i})_i$ are independent exponential random variables of parameter $\alpha^{-i}$, and
	\begin{equation}\label{eq:nu_pi2}
		\nu_\infty(\pi):=p\left(\alpha^{-\pi^{-1}(1)},\dots, \alpha^{-\pi^{-1}(k)}\right)+(1-2k)p\left(\alpha^{-\pi^{-1}(1)},\dots, \alpha^{-\pi^{-1}(k)}\right)^2+2 \sum_{h=1}^{k-1}\eta_{\pi}(h,k),
	\end{equation}
	then, as $n\to\infty$,
	\begin{equation*}
		(\nu_{n}(\pi))^2\sim n\cdot \nu_\infty(\pi).
	\end{equation*}
\end{lem}

\begin{proof}
	It is enough to note that for all $k\geq 1$ and $\pi\in\mcl S_k$,
	\begin{equation*}
		p\left(\theta_{i+\pi^{-1}(1)},\dots, \theta_{i+\pi^{-1}(k)}\right)=p\left(\alpha^{-\pi^{-1}(1)},\dots, \alpha^{-\pi^{-1}(k)}\right).
	\end{equation*}
	Then, the first part of the statement immediately follows. The second part of the statement follows from the additional observation that for all $n\geq 1$,
	\begin{multline*}
		\BB P\left(E'_{i+\pi^{-1}(1)}<\dots<E'_{i+\pi^{-1}(k)}, E'_{i+h+\pi^{-1}(1)}<\dots<E'_{i+h+\pi^{-1}(k)}\right)\\
		=\BB P\left(E_{\pi^{-1}(1)}<\dots<E_{\pi^{-1}(k)}, E_{h+\pi^{-1}(1)}<\dots<E_{h+\pi^{-1}(k)}\right), 
	\end{multline*}
	where $(E'_i)_i$ are independent exponential random variables of parameter $(\alpha^{n-i+1})_i$ respectively, and $(E_{i})_i$ are independent exponential random variables of parameter $\alpha^{-i}$.
\end{proof}

Therefore, we have the following result.

\begin{thm}
	Let $\sigma_n$ be a Luce-distributed permutation of size $n$ with exponential Sukhatme weights $\theta_i=\alpha^{n-i+1}$ for some $\alpha>0$. Fix $k\geq 1$ and $\pi\in\mcl S_k$. Then
	\begin{equation}
		\sqrt{n} \,\frac{\pcocc(\pi,\sigma_n)-\Lambda(\pi)}{\sqrt{\nu_\infty(\pi)}}\xrightarrow{\mathrm{d}}\mcl{N}(0,1),
	\end{equation}
	where $\Lambda(\pi)=p\left(\alpha^{-\pi^{-1}(1)},\dots, \alpha^{-\pi^{-1}(k)}\right)$ and $\nu_\infty(\pi)$ is defined as in \eqref{eq:nu_pi2}.
\end{thm}

As a consequence, $(\sigma_n)_n$ Benjamini--Schramm converges. In particular, $\sigma_n$ locally behaves as a Luce-distributed permutation with parameters $(\alpha^{-i})_i$.

\medskip

We conclude by showing the explicit values for $\Lambda(\pi)$  (for general $\alpha$ and specialized to $\alpha=2$) for all patterns of size 2:
\begin{equation}\label{eq:mean1}
	\Lambda(12)=\frac{\alpha}{\alpha+1}\stackrel{(\alpha=2)}{=} \frac 2 3 \approx 0.667, \qquad  \qquad \Lambda(21)=\frac{1}{\alpha+1}\stackrel{(\alpha=2)}{=} \frac 1 3 \approx 0.333,
\end{equation}
and all patterns of size 3:
\begin{align}\label{eq:mean2}
	&\Lambda(123)=\frac{\alpha^{3}}{(\alpha+1)\,(\alpha^{2}+\alpha+1)}  \stackrel{(\alpha=2)}{=} \frac{8}{21}  \approx 0.381, \\
	&\Lambda(132)=\frac{\alpha^{2}}{(\alpha+1)\,(\alpha^{2}+\alpha+1)}\stackrel{(\alpha=2)}{=} \frac{4}{21} \approx 0.190, \notag\\
	&\Lambda(213)=\frac{\alpha^{3}}{(\alpha^{2}+1)\,(\alpha^{2}+\alpha+1)}\stackrel{(\alpha=2)}{=} \frac{8}{35}  \approx 0.229, \notag\\
	&\Lambda(312)=\frac{\alpha}{(\alpha^{2}+1)\,(\alpha^{2}+\alpha+1)}\stackrel{(\alpha=2)}{=} \frac{2}{35}  \approx 0.057, \notag\\
	&\Lambda(231)=\frac{\alpha}{(\alpha+1)\,(\alpha^{2}+\alpha+1)}\stackrel{(\alpha=2)}{=} \frac{2}{21} \approx 0.095, \notag\\
	&\Lambda(321)=\frac{1}{(\alpha+1)\,(\alpha^{2}+\alpha+1)}\stackrel{(\alpha=2)}{=} \frac{1}{21}  \approx 0.048;\notag
\end{align}
and the values for $\nu_\infty(\pi)$ (only specialized to $\alpha=2$) for all patterns of size 2  (with some matching numerical simulations shown in \Cref{fig:CLT_local_luce2}):
\begin{equation}\label{eq:var1}
	\nu_\infty(12)=\nu_\infty(21)=\frac{2}{21}\approx 0.095,
\end{equation}
and all patterns of size 3 (with some matching numerical simulations shown in \Cref{fig:CLT_local_luce3}):
\begin{align}\label{eq:var2}
	&\nu_\infty(123)  \stackrel{(\alpha=2)}{=} \frac{6184}{22785}  \approx 0.271, \qquad\qquad\quad\,\,\,
	\nu_\infty(132)\stackrel{(\alpha=2)}{=} \frac{6892}{68355} \approx 0.101, \notag\\
	&\nu_\infty(213)\stackrel{(\alpha=2)}{=} \frac{1496456}{15197595}  \approx 0.098, \qquad\qquad
	\nu_\infty(312)\stackrel{(\alpha=2)}{=} \frac{3802}{68355} \approx 0.056, \\
	&\nu_\infty(231)\stackrel{(\alpha=2)}{=} \frac{635822}{15197595}  \approx 0.042, \qquad\qquad
	\nu_\infty(321)\stackrel{(\alpha=2)}{=} \frac{976}{22785}  \approx 0.043.\notag
\end{align}

\begin{figure}[h!]
	\centering
	\includegraphics[width=0.49\textwidth]{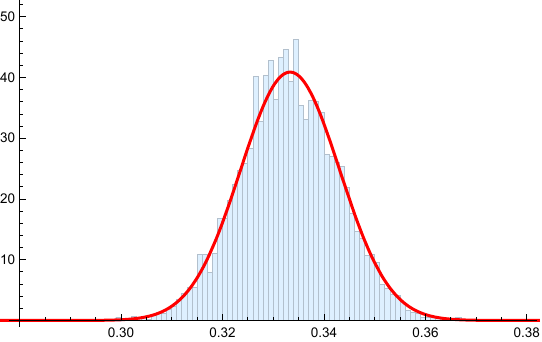}
	\includegraphics[width=0.49\textwidth]{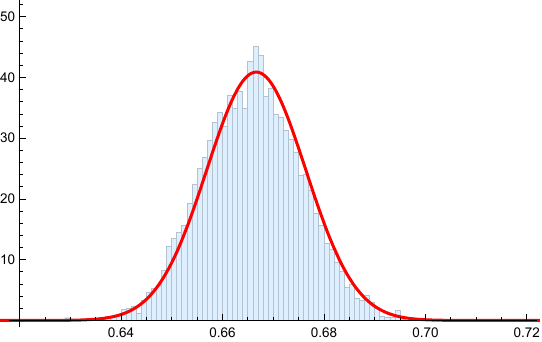}
	\caption{Simulations for the proportion of  consecutive patterns $\pcocc(\pi,\sigma_n)$ when  $\pi\in\mcl S_2$ and $\sigma_n$ is a Luce-distributed permutation of size $n$  with exponential Sukhatme weights $\theta_i=2^{n-i+1}$. In blue we show the histogram (renormalized to be a probability distribution) of the data collected from $3000$ random samples of size $1000$. In red we plot the density of $\mcl{N}(\Lambda(\pi),\nu_\infty(\pi)/1000)$ with the values of $\Lambda(\pi)$ and $\nu_\infty(\pi)$ given in \eqref{eq:mean1} and \eqref{eq:var1}. \textbf{From left to right:} (1) The case of descents, i.e.\ $\pi=21$. (2) The case of $\pi=12$.}
	\label{fig:CLT_local_luce2}
\end{figure}

\begin{figure}[h!]
	\centering
	\includegraphics[width=0.32\textwidth]{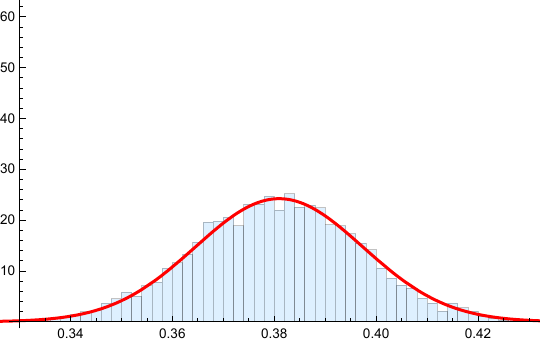}
	\includegraphics[width=0.32\textwidth]{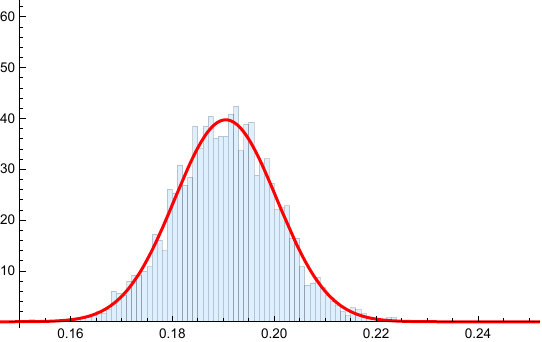}
	\includegraphics[width=0.32\textwidth]{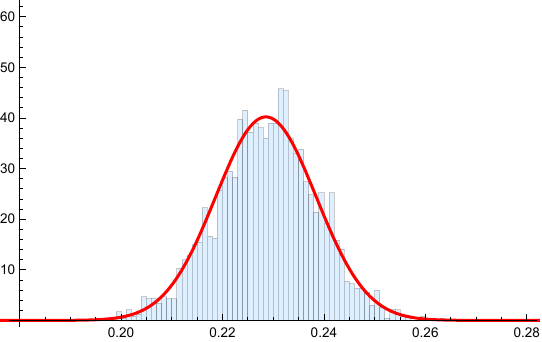}
	
	\includegraphics[width=0.32\textwidth]{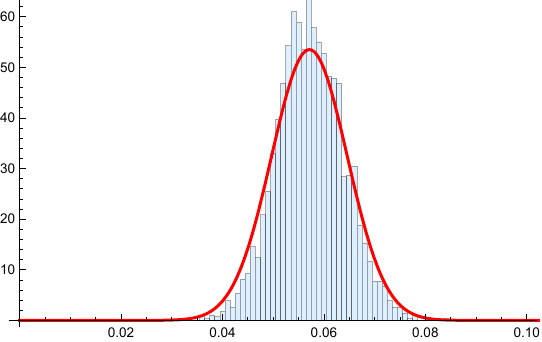}
	\includegraphics[width=0.32\textwidth]{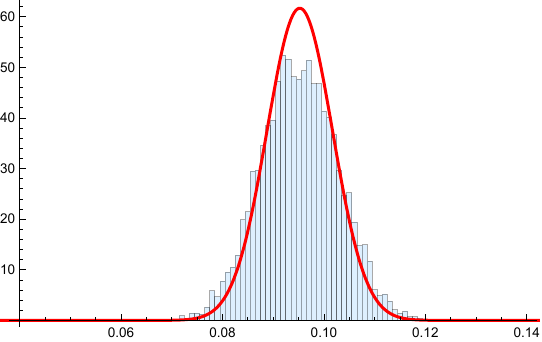}
	\includegraphics[width=0.32\textwidth]{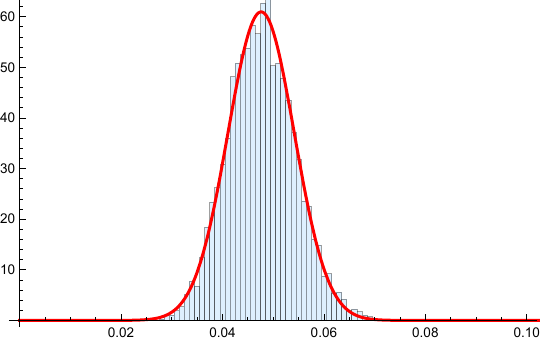}
	
	\caption{Simulations for the proportion of  consecutive patterns $\pcocc(\pi,\sigma_n)$ when  $\pi\in\mcl S_3$ and $\sigma_n$ is a Luce-distributed permutation of size $n$  with exponential Sukhatme weights $\theta_i=2^{n-i+1}$. In blue we show the histogram (renormalized to be a probability distribution) of the data collected from $3000$ random samples of size $1000$. In red we plot the density of $\mcl{N}(\Lambda(\pi),\nu_\infty(\pi)/1000)$ with the values of $\Lambda(\pi)$ and $\nu_\infty(\pi)$ given in \eqref{eq:mean2} and \eqref{eq:var2}. \textbf{From top-left to bottom-right:} (1)  $\pi=123$, (2) $\pi=132$, (3) $\pi=213$, (4) $\pi=312$, (5) $\pi=231$, (6) $\pi=321$.}
	\label{fig:CLT_local_luce3}
\end{figure}

\subsection{Proofs of the main results}\label{sect:proofs-local}

We now turn to the proofs of \Cref{thm:lln_local}, \Cref{thm:clt_local}, and \Cref{thm:general-local-thm}.

\begin{proof}[Proof of \Cref{thm:lln_local}]
	Note that by the nature of the Luce model (recall~\cite[Theorem 2.1]{chatterjee2023enumerative} and Remark~\ref{rmk:inv})
	\begin{equation*}
		\cocc(\pi,\sigma_n)=\sum_{i=0}^{n-k}\mathds{1}_{\left\{\pat_{[i+1,i+k]}(\sigma_n)=\pi\right\}}=\sum_{i=0}^{n-k}\mathds{1}_{\left\{E_{i+\pi^{-1}(1)}<\dots<E_{i+\pi^{-1}(k)}\right\}},
	\end{equation*}
	where $(E_i)_i$ are independent exponential random variables of parameter $(\theta_{i})_i$, respectively. As a consequence,
	\begin{multline*}
		\BB E\left[\pcocc(\pi,\sigma_n)\right]=\frac{1}{n}\sum_{i=0}^{n-k}\BB P\left(E_{i+\pi^{-1}(1)}<\dots<E_{i+\pi^{-1}(k)}\right)\\
		=\frac{1}{n}\sum_{i=0}^{n-k}p\left(\theta_{i+\pi^{-1}(1)},\dots, \theta_{i+\pi^{-1}(k)}\right)\to\Lambda(\pi),
	\end{multline*}
	where in the last step, we used our assumption.
	Moreover,
	\begin{multline*}
		\BB E\left[\pcocc(\pi,\sigma_n)^2\right]=\frac{1}{n^2}\sum_{i=0}^{n-k}\BB P\left(E_{i+\pi^{-1}(1)}<\dots<E_{i+\pi^{-1}(k)}\right)\\
		+\frac{2}{n^2}\sum_{0\leq i<j\leq n-k}\BB P\left(E_{i+\pi^{-1}(1)}<\dots<E_{i+\pi^{-1}(k)}, E_{j+\pi^{-1}(1)}<\dots<E_{j+\pi^{-1}(k)}\right).
	\end{multline*}
	Noting that when the indexes $i<j$ satisfy $j-i\geq k$ then the events 
	\[\left\{E_{i+\pi^{-1}(1)}<\dots<E_{i+\pi^{-1}(k)}\right\}\quad\text{and}\quad\left\{E_{j+\pi^{-1}(1)}<\dots<E_{j+\pi^{-1}(k)}\right\}\] 
	are independent, we get that 
	\begin{align*}
		&\BB E\left[\pcocc(\pi,\sigma_n)^2\right]\\
		=&\frac{1}{n^2}\sum_{i=0}^{n-k}\BB P\left(E_{i+\pi^{-1}(1)}<\dots<E_{i+\pi^{-1}(k)}\right)\\
		&+\frac{2}{n^2}\sum_{h=k}^{n-k}\sum_{i=0}^{n-k-h}\BB P\left(E_{i+\pi^{-1}(1)}<\dots<E_{i+\pi^{-1}(k)}\right)\BB P\left( E_{i+h+\pi^{-1}(1)}<\dots<E_{i+h+\pi^{-1}(k)}\right)\\
		&+\frac{2}{n^2}\sum_{h=1}^{k-1}\sum_{i=0}^{n-k-h}\BB P\left(E_{i+\pi^{-1}(1)}<\dots<E_{i+\pi^{-1}(k)}, E_{i+h+\pi^{-1}(1)}<\dots<E_{i+h+\pi^{-1}(k)}\right).
	\end{align*}
	On the other hand,
	\begin{align*}
		&\BB E\left[\pcocc(\pi,\sigma_n)\right]^2\\
		=&\frac{1}{n^2}\sum_{i=0}^{n-k}\BB P\left(E_{i+\pi^{-1}(1)}<\dots<E_{i+\pi^{-1}(k)}\right)^2\\
		&+\frac{2}{n^2}\sum_{h=k}^{n-k}\sum_{i=0}^{n-k-h}\BB P\left(E_{i+\pi^{-1}(1)}<\dots<E_{i+\pi^{-1}(k)}\right)\BB P\left( E_{i+h+\pi^{-1}(1)}<\dots<E_{i+h+\pi^{-1}(k)}\right)\\
		&+\frac{2}{n^2}\sum_{h=1}^{k-1}\sum_{i=0}^{n-k-h}\BB P\left(E_{i+\pi^{-1}(1)}<\dots<E_{i+\pi^{-1}(k)}\right)\BB P\left( E_{i+h+\pi^{-1}(1)}<\dots<E_{i+h+\pi^{-1}(k)}\right).
	\end{align*}
	Therefore, we conclude that
	\begin{align}\label{eq:variance}
		&\Var\left(\pcocc(\pi,\sigma_n)\right)\notag\\
		=&\frac{1}{n^2}\sum_{i=0}^{n-k}\left(\BB P\left(E_{i+\pi^{-1}(1)}<\dots<E_{i+\pi^{-1}(k)}\right)-\BB P\left(E_{i+\pi^{-1}(1)}<\dots<E_{i+\pi^{-1}(k)}\right)^2\right)\notag\\
		&+\frac{2}{n^2}\sum_{h=1}^{k-1}\sum_{i=0}^{n-k-h}\Bigg(  \BB P\left(E_{i+\pi^{-1}(1)}<\dots<E_{i+\pi^{-1}(k)}, E_{i+h+\pi^{-1}(1)}<\dots<E_{i+h+\pi^{-1}(k)}\right)\notag\\
		&\qquad\qquad\quad-\BB P\left(E_{i+\pi^{-1}(1)}<\dots<E_{i+\pi^{-1}(k)}\right)\BB P\left( E_{i+h+\pi^{-1}(1)}<\dots<E_{i+h+\pi^{-1}(k)}\right)\Bigg).
	\end{align}
	Since all the terms in the sums above are upper bounded by 1, we conclude that
	\begin{equation*}
		\Var\left(\pcocc(\pi,\sigma_n)\right)\leq\frac{n+2kn}{n^2}\to 0.
	\end{equation*}
	The theorem statement follows from a standard application of the Second-moment method.
\end{proof}

\begin{proof}[Proof of \Cref{thm:clt_local}]
	Note that from \eqref{eq:variance} we have that
	\begin{align*}
		&\Var\left(\cocc(\pi,\sigma_n)\right)\\
		&=\sum_{j=0}^{n-k}\left(\BB P\left(E_{j+\pi^{-1}(1)}<\dots<E_{j+\pi^{-1}(k)}\right)-\BB P\left(E_{j+\pi^{-1}(1)}<\dots<E_{j+\pi^{-1}(k)}\right)^2\right)\\
		&\,\,\,\,+2\sum_{h=1}^{k-1}\sum_{i=0}^{n-k-h}\Bigg( \BB P\left(E_{i+\pi^{-1}(1)}<\dots<E_{i+\pi^{-1}(k)}, E_{i+h+\pi^{-1}(1)}<\dots<E_{i+h+\pi^{-1}(k)}\right)\\
		&\qquad\qquad\quad-\BB P\left(E_{i+\pi^{-1}(1)}<\dots<E_{i+\pi^{-1}(k)}\right)\BB P\left( E_{i+h+\pi^{-1}(1)}<\dots<E_{i+h+\pi^{-1}(k)}\right)\Bigg)\\
		&=(\nu_{n}(\pi))^2.
	\end{align*}
	Then, the theorem statement follows from applying \cite[Corollary 2]{baldi1989normal}.
\end{proof}

\begin{proof}[Proof of \Cref{thm:general-local-thm}]
	
	Fix $k\in\mathbb{N}$. For a permutation $\pi$ in $\mcl{S}_k$ define
	
	\begin{equation}\label{eq:phi-pi-defn}
		\Phi_\pi(x_1,\dots,x_k)\coloneqq\prod_{j=1}^k
		\frac{x_{\pi(j)}}{x_{\pi(j)}+x_{\pi(j+1)}+\cdots+x_{\pi(k)}}\,,\qquad (x_1,\dots,x_k)\in(0,\infty)^k.
	\end{equation}
	We first prove \eqref{eq:lim-exp-loc}. Since we are proving the result for any permutation $\pi$, recalling the definition of $p(\cdot)$ from \eqref{eq:p-function}, it is sufficient to prove that 
	\begin{equation}\label{eq:wejifbiweubfowe}
		\lim_{n\to\infty}\frac{1}{n}\sum_{i=0}^{n-k}
		\Phi_\pi\!\bigl(\theta_{i+1},\dots,\theta_{i+k}\bigr)
		=\frac{1}{k!}.
	\end{equation}
	
	First, we observe some basic properties of $\Phi_\pi$.
	Each factor in $\Phi_\pi$ is at most $1$, hence $0\le\Phi_\pi\le 1$. Moreover
	$\Phi_\pi$ is $C^1$  on $(0,\infty)^k$ and simultaneous rescaling of all coordinates does not change $\Phi_\pi$.
	In particular, for any $c>0$, \[\Phi_\pi(c,\dots,c)=\prod_{j=1}^k\frac{1}{k-j+1}=\frac{1}{k!}.\]
	
	We now divide the rest of the proof into five main steps (the first four steps are dedicaded to the proof of \eqref{eq:wejifbiweubfowe}, then the final fifth step is dedicated to the proof of \eqref{eq:lim-var-loc} from the theorem statement).
	
	\smallskip
	\noindent\emph{\underline{Step 1: Restrict $f$ on a large set and gain regularity.}}
	Because $f$ is positive and finite a.e., there exists a sequence of sets
	\[
	E_m \coloneqq \Bigl\{y\in[0,1]: \tfrac{1}{m}\le f(y)\le m\Bigr\}
	\]
	with $|E_m|\uparrow 1$ as $m\to\infty$, where $|E_m|$ denotes the Lebesgue measure of $E_m$.
	Fix $m\in\mathbb{N}$ and $\eta>0$. By Lusin's theorem, there exists a compact set
	$K'\subset E_m$ with $|K'|\ge |E_m|-\eta$ on which $f$ is continuous (hence uniformly continuous).
	By Egorov's theorem, there exists a measurable set $K\subset K'$ with
	$|K|\ge |K'|-\eta$ such that $f_n\to f$ uniformly on $K$.
	
	Summarizing, $K$ is measurable, $|K|\ge |E_m|-2\eta$, $f$ is uniformly continuous on $K$, and $f_n\to f$ uniformly on $K$. Also, on $K$ we have
	$\frac{1}{m}\le f\le m$.
	
	\medskip
	\noindent\emph{\underline{Step 2: Good windows and their proportion.}}
	For $n\in\mathbb{N}$ set $y_i:=i/n$ for $i=0,\dots,n$.
	Call $i\in\{0,\dots,n-k\}$ \textbf{good} if
	\[
	y_i,\ y_i+\tfrac{1}{n},\ \dots,\ y_i+\tfrac{k}{n}\ \in K.
	\]
	Let $G_n$ be the set of good indices, and $B_n$ its complement in $\{0,\dots,n-k\}$ ($B_n$ denotes the set of \textbf{bad} indices).
	We claim
	\begin{equation}\label{eq:bad-bound}
		\#B_n \le (k+1)\cdot\#\{0\le j\le n:\ y_j\notin K\}+k.
	\end{equation}
	Indeed, if $i$ is bad, then for some $r\in\{0,\dots,k\}$ we have $y_{i+r}\notin K$.
	Each point $y_j\notin K$ can serve as $y_{i+r}$ for at most $k+1$ different $i$. Finally, there are at most $k$ boundary indices to account for the range constraint $0\le i\le n-k$.

	We would like now to claim that $\frac{1}{n}\#\{0\le j\le n:\ y_j\notin K\}$ converges to the Lebesgue measure of the complement of $K$ in $[0,1]$, but we need some care since we do not know that $\partial K$ has zero Lebesgue measure. Hence we need to take a small detour. By inner regularity, choose a finite union of closed intervals
	$F\subset K$ with $|K\setminus F|<\eta$. Then
	\[
	\#\{0\le j\le n: y_j\notin K\}
	\;\le\; \#\{0\le j\le n: y_j\notin F\}.
	\]
	Since $F$ is a finite union of intervals,
	$\displaystyle \frac{1}{n}\#\{0\le j\le n: y_j\in F\}\to |F|$,
	hence $\displaystyle \limsup_{n\to\infty}\frac{1}{n}\#\{0\le j\le n: y_j\notin F\}\le 1-|F|$.
	Combining with \eqref{eq:bad-bound} gives
	\[
	\limsup_{n\to\infty}\frac{\#B_n}{n}
	\;\le\; (k+1)\,(1-|F|).
	\]
	Because $|F|\ge |K|-\eta\ge |E_m|-3\eta$, we obtain
	\begin{equation}\label{eq:good-density2}
		\limsup_{n\to\infty}\frac{\#B_n}{n}
		\;\le\; (k+1)\,\bigl(1-|E_m|+3\eta\bigr).
	\end{equation}
	
	\medskip
	\noindent\emph{\underline{Step 3: Uniform control of summands on good indices.}}
	Fix a good $i\in G_n$. For $r=1,\dots,k$ write
	\[
	x_{i,r}\ :=\ \theta_{i+r}\ =\ f_n\!\Bigl(y_i+\frac{r}{n}\Bigr).
	\]
	Fix now $\varepsilon'>0$. Since the entire window $\bigl\{y_i+\frac{r}{n}:1\le r\le k\bigr\}\subset K$ by definition of good indices and $f_n\to f$ uniformly on $K$,
	we get that for $n$ large enough,
	\begin{equation}\label{eq:uniform-approx}
		\max_{1\le r\le k}\Bigl|x_{i,r} - f\!\Bigl(y_i+\frac{r}{n}\Bigr)\Bigr|\ \le\ \varepsilon'.
	\end{equation}
	Because $f$ is uniformly continuous on $K$ and $k/n\to 0$, we also get that for $n$ large enough,
	\begin{equation}\label{eq:uni-cont}
		\max_{1\le r\le k}\Bigl|f\!\Bigl(y_i+\frac{r}{n}\Bigr)-f(y_i)\Bigr|\ \le\ \varepsilon'.
	\end{equation}
	Combining \eqref{eq:uniform-approx} and \eqref{eq:uni-cont}, we obtain that for $n$ large enough,
	\begin{equation}\label{eq;lip-bound}
		\max_{1\le r\le k}\bigl|x_{i,r}-f(y_i)\bigr|\ \le\ 2\varepsilon'.
	\end{equation}
	Since $\frac{1}{m}\le f\le m$ on $K$, choosing $\varepsilon' \le \min\{\frac{1}{4m},\,\frac{m}{2}\}$ we obtain, for $n$ large enough,
	\[
	\frac{1}{2m}\ \le\ x_{i,r}\ \le\ 2m,\qquad \text{for all $r=1,\dots,k$}.
	\]
	Hence, for $n$ large enough, the vector $(x_{i,1},\dots,x_{i,k})$ lies in the compact rectangle
	$\mathcal{R}_m:=[\frac{1}{2m},\,2m]^k\subset(0,\infty)^k$.
	Because $\Phi_\pi$ is $C^1$ on $(0,\infty)^k$, it is Lipschitz on $\mathcal{R}_m$:
	There exists $L_m<\infty$ such that
	\[
	\bigl|\Phi_\pi(u)-\Phi_\pi(v)\bigr| \ \le\ L_m\,\|u-v\|_\infty,\qquad \text{ for all $u,v\in\mathcal{R}_m$}.
	\]
	Applying this with $u=(x_{i,1},\dots,x_{i,k})$ and $v=(f(y_i),\dots,f(y_i))$ gives (thanks to \eqref{eq;lip-bound}) the following bound: for all $n$ large enough,
	\begin{equation}\label{eq:lipschitz}
		\Bigl|\,\Phi_\pi\bigl(x_{i,1},\dots,x_{i,k}\bigr) - \Phi_\pi\bigl(f(y_i),\dots,f(y_i)\bigr)\,\Bigr|
		\ \le\ 2L_m\,\varepsilon'.
	\end{equation}
	Since $\Phi_\pi(c,\dots,c)=1/k!$ for all $c>0$,~\eqref{eq:lipschitz} yields, for all good $i\in G_n$ and all $n$ large enough,
	\begin{equation}\label{eq:good-close}
		\Bigl|\,\Phi_\pi\bigl(\theta_{i+1},\dots,\theta_{i+k}\bigr) - \tfrac{1}{k!}\,\Bigr|
		\ \le\ 2L_m\,\varepsilon'.
	\end{equation}
	
	\medskip
	\noindent\emph{\underline{Step 4: Average and limits.}}
	Using $0\le \Phi_\pi\le 1$ and decomposing the average over $G_n$ and $B_n$, we get
	\[
	\left|\frac{1}{n}\sum_{i=0}^{n-k}\Phi_\pi(\theta_{i+1},\dots,\theta_{i+k}) - \frac{1}{k!}\right|
	\ \le\ \frac{1}{n}\sum_{i\in G_n}\Bigl|\,\Phi_\pi(\theta_{i+1},\dots,\theta_{i+k})-\tfrac{1}{k!}\,\Bigr| + \frac{\#B_n}{n}.
	\]
	By \eqref{eq:good-close}, for $n$ large enough,
	\[
	\frac{1}{n}\sum_{i\in G_n}\Bigl|\,\Phi_\pi(\theta_{i+1},\dots,\theta_{i+k})-\tfrac{1}{k!}\,\Bigr| \ \le\ 2L_m\,\varepsilon'\,\frac{\#G_n}{n}
	\ \le\ 2L_m\,\varepsilon',
	\]
	since $G_n\subset\{0,\dots,n-k\}$ by definition. Taking $\varepsilon' := \min\{\frac{1}{4m},\,\frac{m}{2},\,\frac{\varepsilon}{4L_m}\}$ gives
	\[
	\limsup_{n\to\infty}
	\left|\frac{1}{n}\sum_{i=0}^{n-k}\Phi_\pi(\theta_{i+1},\dots,\theta_{i+k}) - \frac{1}{k!}\right|
	\ \le\ \frac{\varepsilon}{2} + \limsup_{n\to\infty}\frac{\#B_n}{n}.
	\]
	By \eqref{eq:good-density2}, we get that 
	\[
	\limsup_{n\to\infty}
	\left|\frac{1}{n}\sum_{i=0}^{n-k}\Phi_\pi(\theta_{i+1},\dots,\theta_{i+k}) - \frac{1}{k!}\right|
	\ \le\ \frac{\varepsilon}{2} + (k+1)\,(1-|E_m|+3\eta).
	\]
	Now, first let $\eta\downarrow 0$, then let $m\to\infty$ so that $|E_m|\uparrow 1$, and finally, let $\varepsilon\downarrow 0$.
	This concludes the proof of \eqref{eq:wejifbiweubfowe}.

	\medskip
	\noindent\emph{\underline{Step 5: The limiting variance.}} We finally prove the claim in \eqref{eq:lim-var-loc}.
	Set for all $h\in\{1,\ldots,k-1\}$ and $(x_1,\dots,x_{h+k})\in(0,\infty)^{h+k}$,
	\begin{equation*}
		J^{h}_{\pi}(x_1,\dots,x_{h+k})\coloneqq\BB P\left(E_{\pi(1)}<\dots<E_{\pi(k)}, E_{h+\pi(1)}<\dots<E_{h+\pi(k)}\right),
	\end{equation*}
	where $(E_i)_i$ are independent exponential random variables of parameter $(x_{i})_i$. Recall also the function $\Phi_\pi$ introduced in \eqref{eq:phi-pi-defn}.

	Again, since we are proving the result for any permutation $\pi$, it is sufficient to prove that (recall the expression for $(\nu_{n}(\pi))^2$ from \eqref{eq:defn-variance} in \Cref{thm:clt_local})
	\begin{multline*}
		\frac{1}{n}\sum_{i=0}^{n-k}\left(\Phi_\pi\!\bigl(\theta_{i+1},\dots,\theta_{i+k}\bigr)-\Phi_\pi\!\bigl(\theta_{i+1},\dots,\theta_{i+k}\bigr)^2\right)\\
		+\frac{2}{n}\sum_{h=1}^{k-1}\sum_{i=0}^{n-k-h}\left(J^{h}_{\pi}\bigl(\theta_{i+1},\dots,\theta_{i+h+k}\bigr)-\Phi_\pi\!\bigl(\theta_{i+1},\dots,\theta_{i+k}\bigr)\Phi_\pi\!\bigl(\theta_{i+h+1},\dots,\theta_{i+h+k}\bigr)\right)\\
		\longrightarrow \frac{1}{k!}+\frac{1-2k}{(k!)^2}+2\sum_{h=1}^{k-1}\zeta_{\pi^{-1}}(h,k),
	\end{multline*}
	where  $\zeta_\pi(h,k)$ was introduced in \eqref{eq:wefvbiwebfow}.

	Note that for all $c>0$,
	\begin{equation*}
		\Phi_\pi(c,\ldots,c)=1/k!\qquad\text{and}\qquad J_\pi^h(c,\ldots,c)=\zeta_\pi(h,k).
	\end{equation*}
	Moreover, for every $m\in\mathbb N$, there exists $L_m<\infty$ such that for all $h\in\{1,\ldots,k-1\}$,
	\begin{align}
		&|\Phi_\pi(u)-\Phi_\pi(v)|
		\le L_m\|u-v\|_\infty
		\quad\forall\,u,v\in[1/m,m]^k,\notag\\
		&|J^{h}_{\pi}(u)-J^{h}_{\pi}(v)|
		\le L_m\|u-v\|_\infty
		\quad\;\forall\,u,v\in[1/m,m]^{h+k}.\label{eq;lewvbeiw}
	\end{align}
	Indeed, the first inequality was already used in Step 3 of the proof and for the second inequality it is enought to notice the following facts: for $u=(u_1,\dots,u_{h+k})\in[1/m,m]^{h+k}$,
	\[
	J^h_\pi(u)=\int_A \Big(\prod_{r=1}^{h+k}u_r\Big)\exp\!\Big(-\sum_{r=1}^{h+k} u_r t_r\Big)\,dt,
	\]
	where $A\subset(0,\infty)^{h+k}$ is a polyhedral cone.
	Differentiating the integrand, we obtain
	\[
	\frac{\partial}{\partial u_r}
	\Big[\prod_s u_s\,e^{-\sum_su_s t_s}\Big]
	=\Big(\tfrac{1}{u_r}-t_r\Big)\prod_s u_s\,e^{-\sum_su_s t_s}.
	\]
	For $u_r\in[1/m,m]$, this is bounded by  $C_m(1+t_r)e^{-1/m\sum_s t_s}$, which is integrable on $(0,\infty)^{h+k}$, uniformly in $u$. By dominated convergence, $J^h_\pi$ is $C^1$ on $[1/m,m]^{h+k}$ and $\nabla J^h_\pi$ is bounded. Taking $L_m:=\sup_{h\in\{1,\dots,k-1\},\xi\in[1/m,m]^{h+k}}\|\nabla J^h_\pi(\xi)\|_1$ proves \eqref{eq;lewvbeiw}.

	Since we have the same properties for $J^h_\pi$ as we had for $\Phi_\pi$, using exactly the same proof we used to prove in the previous steps that
	\[
	\lim_{n\to\infty}\frac{1}{n}\sum_{i=0}^{n-k}
	\Phi_\pi\!\bigl(\theta_{i+1},\dots,\theta_{i+k}\bigr)
	=\frac{1}{k!},
	\]
	we get that 
	\[
	\frac{1}{n}\sum_{i=0}^{n-k}\left(\Phi_\pi\!\bigl(\theta_{i+1},\dots,\theta_{i+k}\bigr)-\Phi_\pi\!\bigl(\theta_{i+1},\dots,\theta_{i+k}\bigr)^2\right)
	\longrightarrow \frac{1}{k!}-\frac{1}{(k!)^2},
	\]
	and that
	\begin{multline*}
		\frac{2}{n}\sum_{h=1}^{k-1}\sum_{i=0}^{n-k-h}\left(J^{h}_{\pi}\bigl(\theta_{i+1},\dots,\theta_{i+h+k}\bigr)-\Phi_\pi\!\bigl(\theta_{i+1},\dots,\theta_{i+k}\bigr)\Phi_\pi\!\bigl(\theta_{i+h+1},\dots,\theta_{i+h+k}\bigr)\right)\\
		\longrightarrow 2\sum_{h=1}^{k-1}\zeta_{\pi^{-1}}(h,k)-\frac{2k-2}{(k!)^2}.
	\end{multline*}
	The sum of the right hand side of the last two equations is
	\[
	\frac{1}{k!}+\frac{1-2k}{(k!)^2}+2\sum_{h=1}^{k-1}\zeta_{\pi^{-1}}(h,k),
	\]
	as we wanted. This concludes the proof of \eqref{eq:lim-var-loc}, and thereby the entire theorem.
\end{proof}

\section{A central limit theorem for the number of inversions}\label{sect:clt-inv}

\Cref{sect:local} developed central limit theorems for local statistics under the Luce model. In this section, we prove a similar theorem for global patterns, at least for the most well-known special case: inversions. The argument is a variant of the Hoeffding-Hajek projection method.
We believe it generalizes to more complex patterns.

\begin{thm}\label{them:clt-inv}
	Let $\theta_1, \theta_2, \ldots,\theta_n$ be positive real numbers. Let $\sigma_n\sim \Luce(\theta_1,\dots,\theta_n)$ and recall that $\occ(21,\sigma_n)$ denotes the number of inversions in $\sigma_n$.
	Then 
	\[
	\bbE(\occ(21,\sigma_n)) = \sum_{1\le i<j\le n} \frac{\theta_j}{\theta_i+\theta_j},
	\]
	and 
	\[
	\Var(\occ(21,\sigma_n)) = \sum_{1\le i<j\le n} \frac{\theta_i\theta_j}{(\theta_i+\theta_j)^2} + \sum_{1\le i<j<k\le n} \frac{4\theta_i\theta_j^2\theta_k}{(\theta_i+\theta_j+\theta_k)(\theta_i+\theta_j)(\theta_i+\theta_k)(\theta_j+\theta_k)}. 
	\]
	Lastly, define $a := \left(\sqrt{\Var(\occ(21,\sigma_n))} - n \right)^+$, where $x^+$ denotes the positive part of a real number $x$. 
	Then
	\begin{align*}
		\sup_{t\in \bbR} \biggl|\bbP\biggl(\frac{\occ(21,\sigma_n)-\bbE(\occ(21,\sigma_n))}{\sqrt{\Var(\occ(21,\sigma_n))}} \le t\biggr) - \Phi(t)\biggr|&\le \frac{Cn^4}{a^{3}} +  \frac{Cn}{a} + \frac{Cn^{2/3}}{a^{2/3}},
	\end{align*}
	where $C$ is a universal constant and $\Phi(t)$ denotes the cumulative distribution function of the standard normal distribution.
\end{thm}

The above theorem shows that as long as $\Var(Z) \gg n^{8/3}$, $Z$ is approximately normally distributed. In typical situations, where all the $\theta_i$'s are all of comparable size, the above formula for $\Var(Z)$ shows that it is of order $n^3$, and hence this condition is satisfied. This holds, for example, for the Sukhatme weights as well as for uniform random permutations.

\medskip

In particular, for the Sukhatme weights \(\theta_i = n - i + 1\), using  the standard asymptotic expansion for harmonic numbers
\[\sum_{k=1}^{m} \frac{1}{k} = \log m + \gamma + O\left(\frac{1}{m}\right),\] 
and standard Riemann sum approximations, we get 
\[
\mathbb{E}[\mathrm{occ}(21, \sigma_n)] = n^2 \int_0^1 t \log \frac{1 + t}{2t} \, dt + O(n \log n) = \frac{1 - \log 2}{2} n^2 + O(n \log n).
\]
and
\[
\mathrm{Var}(\mathrm{occ}(21, \sigma_n)) = V \cdot n^3 + O(n^2 \log n),
\]
where (the final approximation is  a numerical approximation)
\[
V=\int_{0<t<u<v<1} \frac{4\,t\,u^{2}\,v}{(t+u+v)(t+u)(t+v)(u+v)} \, dt\,du\,dv\approx 0.0181166.
\]
Some numerical simulations for the result proved in \Cref{them:clt-inv} in this specific setting are shown in~\Cref{fig:CLT_inv_luce}.

\begin{figure}[h!]
	\centering
	\includegraphics[width=0.6\textwidth]{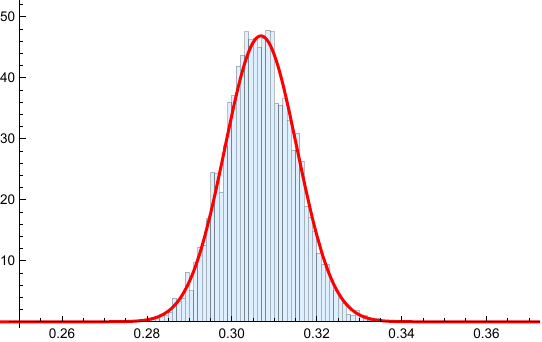}
	
	\caption{Simulations for the proportion of  inversions $\pocc(21,\sigma_n)=\frac{\occ(21,\sigma_n)}{\binom{n}{2}}$ when $\sigma_n$ is a Luce-distributed permutation of size $n$  with Sukhatme weights $\theta_i=n-i+1$. In blue we show the histogram (renormalized to be a probability distribution) of the data collected from $3000$ random samples of size $n=1000$. In red we plot the density of $\mcl{N}(1-\log(2),\frac{4 \times 0.0181166}{n})$.}
	\label{fig:CLT_inv_luce}
\end{figure}

\begin{proof}[Proof of \Cref{them:clt-inv}]
	
	Let $E_1,\ldots,E_n$ be independent random variables, with $E_i$ having an exponential distribution with mean $1/\theta_i$. Let
	\[
	Z := \sum_{1\le i<j\le n} 1_{\{E_i > E_j\}}.
	\]\
	Then, as usual, $\occ(21,\sigma_n)\stackrel{\mathrm{d}}{=}Z$ by the exponential representation of the Luce model (see for instance~\cite[Theorem 2.1]{chatterjee2023enumerative}).
	Define
	\begin{align*}
		f_{ij} &:= 1_{\{E_i > E_j\}}, \\ 
		g_{ij} &:= \bbE(f_{ij}|E_i) = 1 - e^{-\theta_j E_i}, \\ 
		h_{ij} &:= \bbE(f_{ij}|E_j) = e^{-\theta_i E_j}, \\
		\mu_{ij} &:= \bbE(f_{ij}) = \frac{\theta_j}{\theta_i+\theta_j}. 
	\end{align*}
	Then note that 
	\begin{align*}
		\bbE(Z) &= \sum_{1\le i< j\le n}\mu_{ij} = \sum_{1\le i< j\le n} \frac{\theta_j}{\theta_i+\theta_j}. 
	\end{align*}
	Next, note that
	\begin{align*}
		\Var(Z) &= \sum_{i<j,\, i'<j'}\mathrm{Cov}(f_{ij}, f_{i'j'}). 
	\end{align*}
	Take any $i<j$ and $i'<j'$. If $\{i,j\}\cap \{i',j'\} =\emptyset$, then clearly, $\mathrm{Cov}(f_{ij}, f_{i'j'})=0$. If $|\{i,j\}\cap \{i',j'\}|=2$, then $i=i'$ and $j=j'$. In this case, the covariance is $\mu_{ij}(1-\mu_{ij})$. Lastly, suppose that $|\{i,j\}\cap \{i',j'\}|=1$. There are several possibilities. First, we may have $i=i'<j<j'$. In this case,
	\begin{align*}
		\bbE(f_{ij}f_{i'j'}) &= \bbP(E_i > E_j> E_{j'})  + \bbP(E_i > E_{j'} > E_j) \\
		&= \frac{\theta_{j'} \theta_j }{(\theta_i +\theta_j+\theta_{j'})(\theta_i +\theta_j)}+ \frac{\theta_{j} \theta_{j'} }{(\theta_i +\theta_j+\theta_{j'})(\theta_i +\theta_{j'})}.
	\end{align*}
	Thus,
	\begin{align*}
		\mathrm{Cov}(f_{ij}, f_{i'j'}) &= \bbE(f_{ij}f_{i'j'}) - \bbE(f_{ij})\bbE(f_{i'j'})\\
		&= \frac{\theta_{j'} \theta_j }{(\theta_i +\theta_j+\theta_{j'})(\theta_i +\theta_j)}+ \frac{\theta_{j} \theta_{j'} }{(\theta_i +\theta_j+\theta_{j'})(\theta_i +\theta_{j'})} - \frac{\theta_j\theta_{j'}}{(\theta_i+\theta_j)(\theta_i + \theta_{j'})}\\
		&= \frac{\theta_i\theta_j \theta_{j'}}{(\theta_i +\theta_j+\theta_{j'})(\theta_i +\theta_j)(\theta_i +\theta_{j'})}.
	\end{align*}
	The same holds if $i=i' < j' < j$. Next, suppose that $i < j = i' < j'$. In this case,
	\begin{align*}
		\mathrm{Cov}(f_{ij}, f_{i'j'}) &= \bbP(E_i > E_j> E_{j'}) - \bbP(E_i>E_j)\bbP(E_j>E_{j'}) \\
		&= \frac{\theta_{j'} \theta_j }{(\theta_i +\theta_j+\theta_{j'})(\theta_i +\theta_j)} - \frac{\theta_j \theta_{j'}}{(\theta_i + \theta_j)(\theta_j + \theta_{j'})}\\
		&= -\frac{\theta_i\theta_j\theta_{j'}}{(\theta_i + \theta_j + \theta_{j'})(\theta_i +\theta_j)(\theta_j + \theta_{j'})}.
	\end{align*}
	The next possibility is $i' < i < j=j'$. In this case
	\begin{align*}
		\bbE(f_{ij} f_{i'j'}) &= \bbP(E_{i'} > E_i> E_j) + \bbP(E_i> E_{i'} > E_j)\\
		&= \frac{\theta_j \theta_i}{(\theta_{i'} + \theta_i + \theta_j)(\theta_{i'}+\theta_i)} + \frac{\theta_j \theta_{i'}}{(\theta_{i'} + \theta_i + \theta_j)(\theta_{i'}+\theta_i)}\\
		&= \frac{\theta_j}{\theta_{i'}+\theta_i+\theta_j}. 
	\end{align*}
	Thus,
	\begin{align*}
		\mathrm{Cov}(f_{ij}, f_{i'j'}) &= \frac{\theta_j}{\theta_{i'}+\theta_i+\theta_j} - \frac{\theta_j^2}{(\theta_i+\theta_j)(\theta_{i'}+\theta_j)}\\
		&= \frac{\theta_i\theta_{i'}\theta_j}{(\theta_{i'}+\theta_i+\theta_j) (\theta_i+\theta_j)(\theta_{i'}+\theta_j)}.
	\end{align*}
	The same holds if $i<i'<j=j'$. Finally, the last possibility is $i'<j'=i<j$. This is symmetrical with the case $i<j=i'<j'$, which shows that
	\begin{align*}
		\mathrm{Cov}(f_{ij}, f_{i'j'}) &= -\frac{\theta_{i'}\theta_{j'}\theta_{j}}{(\theta_{i'} + \theta_{j'} + \theta_{j})(\theta_{i'} +\theta_{j'})(\theta_{j'} + \theta_{j})}.
	\end{align*}
	Combining all cases, we get
	\begin{align}\label{eq:varZ}
		\Var(Z) &= \sum_{i<j} \mu_{ij}(1-\mu_{ij}) + \sum_{i<j<k} \frac{2\theta_i\theta_j\theta_k}{\theta_i+\theta_j+\theta_k}\biggl(\frac{1}{(\theta_i+\theta_j)(\theta_i+\theta_k)} \notag\\
		&\qquad \qquad \qquad \qquad \qquad \qquad - \frac{1}{(\theta_i+\theta_j)(\theta_j+\theta_k)} + \frac{1}{(\theta_i+\theta_k)(\theta_j+\theta_k)} \biggr)\notag\\
		&= \sum_{i<j} \frac{\theta_i\theta_j}{(\theta_i+\theta_j)^2} + \sum_{i<j<k} \frac{4\theta_i\theta_j^2\theta_k}{(\theta_i+\theta_j+\theta_k)(\theta_i+\theta_j)(\theta_i+\theta_k)(\theta_j+\theta_k)}. 
	\end{align}
	Now, let
	\begin{align*}
		\eta_{ij} := f_{ij}-g_{ij}-h_{ij}+\mu_{ij}. 
	\end{align*}
	Clearly, $\bbE(\eta_{ij})=0$. If $\{i,j\}\cap \{i',j'\}=\emptyset$, then $\eta_{ij}$ and $\eta_{i'j'}$ are independent, and so, $\bbE(\eta_{ij}\eta_{i'j'})=0$. We claim that $\bbE(\eta_{ij}\eta_{i'j'})=0$ even if $|\{i,j\}\cap\{i',j'\}|=1$. To prove this, consider the case $i=i'<j<j'$. Then 
	\begin{align*}
		\bbE(\eta_{ij}|(E_k)_{k\ne j}) &= \bbE(f_{ij}|(E_k)_{k\ne j}) - g_{ij} - \bbE(h_{ij}|(E_k)_{k\ne j}) + \mu_{ij}\\
		&= \bbE(f_{ij}|E_i) - g_{ij} - \bbE(h_{ij})+\mu_{ij} = 0.
	\end{align*}
	Since $\eta_{i'j'}$ is a function of $(E_k)_{k\ne j}$ in this case, this proves that $\bbE(\eta_{ij}\eta_{i'j'})=0$. The other cases are proved similarly. Thus, we conclude that 
	\begin{align}\label{etaineq}
		\bbE\biggl[\biggl(\sum_{i<j} \eta_{ij}\biggr)^2\biggr] = \sum_{i<j} \bbE(\eta_{ij}^2)\le n^2. 
	\end{align}
	But,
	\begin{align*}
		\sum_{i<j} \eta_{ij} &= Z- \bbE(Z) - W,
	\end{align*}
	where 
	\begin{align*}
		W &= \sum_{i<j} (g_{ij}+h_{ij}-2\mu_{ij})= \sum_{i=1}^n W_i,
	\end{align*}
	with 
	\[
	W_i := \sum_{j=1}^{i-1} (e^{-\theta_j E_i}-\mu_{ji}) -\sum_{j=i+1}^n (e^{-\theta_j E_i} - \mu_{ji}) . 
	\]
	Note that $W_1,\ldots,W_n$ are independent centered random variables, and $|W_i|\le n$ for each $i$. Thus, by the Berry--Esseen theorem for sums of independent and non-identically distributed centered random variables,
	\begin{align*}
		\sup_{t\in \bbR} \biggl|\bbP\biggl(\frac{W}{\sqrt{\Var(W)}}\le t\biggr) - \Phi(t)\biggr|&\le \frac{\sum_{i=1}^n \bbE|W_i|^3}{2(\Var(W))^{3/2}}\le \frac{n^4}{2(\Var(W))^{3/2}},
	\end{align*}
	where we recall that $\Phi$ denotes the standard normal cumulative distribution function. Thus, by \eqref{etaineq}, we have that for any $t\in \bbR$ and $\varepsilon >0$,
	\begin{align*}
		&\bbP(Z\le \bbE(Z) + t\sqrt{\Var(Z)}) \\
		&\le \bbP(W\le (t+\varepsilon)\sqrt{\Var(Z)}) + \bbP(|Z-\bbE(Z)-W|> \varepsilon \sqrt{\Var(Z)})\\
		&\le \Phi\biggl(\frac{(t+\varepsilon)\sqrt{\Var(Z)}}{\sqrt{\Var(W)}}\biggr) + \frac{n^4}{2(\Var(W))^{3/2}} + \frac{n^2}{\varepsilon^2\Var(Z)}\\
		&\le \Phi(t) +a(t)\biggl|\frac{\sqrt{\Var(Z)}}{\sqrt{\Var(W)}} - 1\biggr|+ \frac{\varepsilon\sqrt{\Var(Z)}}{2\sqrt{\Var(W)}}+ \frac{n^4}{2(\Var(W))^{3/2}} + \frac{n^2}{\varepsilon^2\Var(Z)},
	\end{align*}
	where
	\[
	a(t) := |t| \max\biggl\{|\Phi'(s)|: s \text{ lies between } t \text{ and } \frac{t \sqrt{\Var(Z)}}{\sqrt{\Var(W)}}\biggr\}. 
	\]
	Optimizing over $\varepsilon>0$, and applying \eqref{etaineq} through the inequality 
	\[
	|\sqrt{\Var(Z)} - \sqrt{\Var(W)}|\le \sqrt{\Var(Z-W)} \le n,
	\]
	we get
	\begin{align}\label{semifinal}
		&\bbP(Z\le \bbE(Z) + t\sqrt{\Var(Z)}) \notag \\
		&\le \Phi(t) + a(t) \biggl|\frac{\sqrt{\Var(Z)}}{\sqrt{\Var(W)}} - 1\biggr| + \frac{n^4}{2(\Var(W))^{3/2}} + \frac{C_0 n^{2/3}}{(\Var(W))^{1/3}}\notag\\
		&\le \Phi(t) + \frac{a(t)n}{\sqrt{\Var(W)}}  + \frac{n^4}{2(\Var(W))^{3/2}} + \frac{C_0 n^{2/3}}{(\Var(W))^{1/3}},
	\end{align}
	where $C_0$ is a universal constant. 
	Next, by \eqref{etaineq}, we have 
	\begin{align*}
		\bbE[(Z-\bbE(Z)-W)^2] &= \sum_{i<j} \bbE(\eta_{ij}^2)= \sum_{i<j} \Var(f_{ij}-g_{ij}-h_{ij})\\
		&\le 3\sum_{i<j} (\Var(f_{ij})+\Var(g_{ij}) + \Var(h_{ij}))\\
		&\le 9\sum_{i<j} \Var(f_{ij}),
	\end{align*}
	where, in the second inequality, we used the law of total variance.
	Now recall from \eqref{eq:varZ} that
	\[
	\Var(Z)
	= \sum_{i<j}\Var(f_{ij})
	+ \sum_{i<j<k}\frac{4\theta_i\theta_j^2\theta_k}{(\theta_i+\theta_j+\theta_k)(\theta_i+\theta_j)(\theta_i+\theta_k)(\theta_j+\theta_k)}.
	\]
	Since the second term is nonnegative, we conclude that
	\[
	\sum_{i<j}\Var(f_{ij}) \le \Var(Z),
	\]
	and so
	\begin{align*}
		\bbE[(Z-\bbE(Z)-W)^2] \le 9 \Var(Z). 
	\end{align*}
	Thus,
	\begin{align*}
		\Var(W) &\le 2\Var(Z) + 2\Var(Z-W)\le 11 \Var(Z). 
	\end{align*}
	A consequence of this inequality is that $a(t)$ is bounded above by a universal constant. Thus, by~\eqref{semifinal}, we get
	\begin{align*}
		\bbP(Z\le \bbE(Z) + t\sqrt{\Var(Z)}) &\le \Phi(t) + \frac{n^4}{2(\Var(W))^{3/2}} \\
		&\qquad + \frac{C_1n}{\sqrt{\Var(W)}}+ \frac{C_0n^{2/3}}{(\Var(W))^{1/3}},
	\end{align*}
	where $C_1$ is another universal constant. A similar argument shows that
	\begin{align*}
		\bbP(Z\le \bbE(Z) + t\sqrt{\Var(Z)}) &\ge \Phi(t) - \frac{n^4}{2(\Var(W))^{3/2}} \\
		&\qquad - \frac{C_1n}{\sqrt{\Var(W)}}- \frac{C_0n^{2/3}}{(\Var(W))^{1/3}}. 
	\end{align*}
	Combining these two bounds, and applying \eqref{etaineq} via the inequality 
	\[
	\sqrt{\Var(W)} \ge \sqrt{\Var(Z)} - \sqrt{\Var(Z-W)}\ge \sqrt{\Var(Z)} - n
	\]
	completes the proof.
\end{proof}

\section{Comments and open questions}\label{sect:final}

The Luce model is simply the most basic model of sampling from an urn without replacement. Thus, any reasonable question is worth studying. Extensive references in this direction are in \cite{chatterjee2023enumerative}.

\subsection{Statistical applications}

We have not emphasized it here, but as an $n - 1$ parameter statistical model, the Luce model is frequently fit to permutation data. 

Of course, other models are used as well. The book by Marden~\cite{marden1996analyzing} is a useful source. This is grist for our mill. How can one test whether the Luce model is better than the Mallows model? Here, if $d(\cdot, \cdot)$ is a metric on $\mcl S_n$, the Mallows model is:
\[
\bbP_{\beta,\sigma_0}(\sigma) = \frac{\mathrm{e}^{-\beta \cdot d(\sigma, \sigma_0)}}{Z}
\]
with $\beta$ a scale parameter and $\sigma_0$ a location parameter to be estimated from the data ($Z$ is a normalization constant).

In a fascinating study~\cite{Alimohammadi}, the Luce and Mallows models were each fitted to basketball ranking data. Mallows seemed better suited for predicting new data. This is surprising since the Luce model has $n-1$ free parameters, and the version of Mallows used had only two free parameters (they used an $\ell^p$-metric and fit $p$, $\beta$, and $\sigma_0$). But Mallows has a discrete parameter $\sigma_0$. How should that be factored in?
This is \textbf{not} a standard statistical testing problem. One way to distinguish models is to look at ``features'', such as descents and inversions, where limit theory can be used. This motivates the development of new limit theorems for both models. Our paper contributes to that. The problems below show how much remains to be understood.

\subsection{Cycle structure}

Here is a frustrating open problem for the Luce model: Pick $\sigma$ from the Luce distribution on $S_n$. What is the limiting distribution of $c_1(\sigma)$ -- the number of fixed points of $\sigma$, in the presence of a permuton limit $\mu$ (with density $\rho$) for $\sigma$?

It is natural to conjecture that $c_1(\sigma)$ should be approximately $\text{Poisson}(\theta)$, with
\[
\theta = \int_0^1 \rho(x,x) \, dx,
\]
when the above integral is finite. But one can immediately see that for the standard Sukhatme weights $\theta_i=n-i+1$, the density computed in \eqref{eq:suk-dens} (recall that it is singular at $(1,1)$) leads to an infinite integral. In this case, we believe the number of fixed points tends to infinity, but do the number of fixed points still converge to the Poisson distribution after an appropriate rescaling? Some numerical simulations are shown in \Cref{fig:fixed-points}.

Similarly, the joint distribution of $\{c_i(\sigma)\}_i$, where $c_i(\cdot)$ denotes the number of $i$-cycles, and even the total number of cycles, is open.

\begin{figure}[h!]
	\centering
	\includegraphics[width=0.49\textwidth]{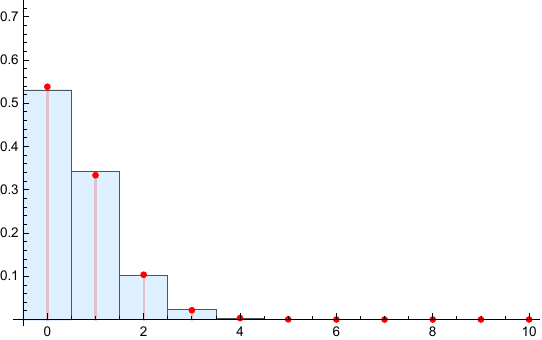}
	\includegraphics[width=0.49\textwidth]{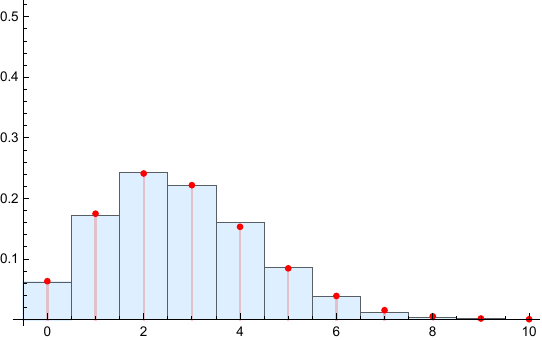}
	\includegraphics[width=0.49\textwidth]{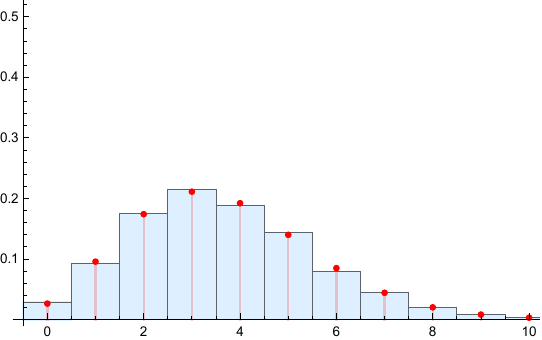}
	\includegraphics[width=0.49\textwidth]{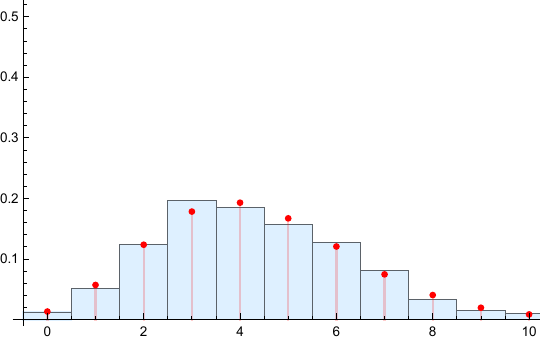}
	
	\caption{Simulations for the number of fixed points $c_1(\sigma_n)$ when $\sigma_n$ is a Luce-distributed permutation of size $n$. \textbf{Diagram on the top-left:} Here the weights are $\theta_i=\frac{i}{n}$, and so, thanks to Proposition~\ref{prop:density}, the corresponding limiting permuton has density $\rho(1-x,y)$ with $\rho(\cdot,\cdot)$ as in \eqref{eq:suk-dens}. In particular, $\int_0^1 \rho(1-x,x) \,\mathrm{d}x\approx0.629$. In blue, we show the histogram (renormalized to be a probability distribution) of the data collected from $3000$ random samples of size $n=1000$. In red we plot the distribution function of a $\text{Poisson}(0.629)$. \textbf{Remaining three diagrams:} Here the weights are the standard Sukhatme weights $\theta_i=\frac{n-i+1}{n}$, and so, thanks to Proposition~\ref{prop:density}, the limiting permuton has density $\rho(x,y)$ as shown in \eqref{eq:suk-dens}. In particular, $\int_0^1 \rho(x,x) = +\infty$ in this case. In the three diagrams (top-right/bottom-left/bottom-right), we show the histograms (renormalized to be a probability distribution) of the data collected from $4000/3000/2000$ random samples of size $n=1000/10000/100000$. Note that the number of fixed points grows with $n$. In red we plot the distribution functions of the $\text{Poisson}(\theta)$ with $\theta=2.76/3.64/4.33$ (equal to the mean of the data). The data seems to suggest that the number of fixed points is still Poisson distributed after appropriate rescaling.}
	\label{fig:fixed-points}
\end{figure}

\subsection{Inverses and longest increasing subsequences}

For a $\sigma \in \mcl S_n$, if $\sigma(i)$ is the label of the card at position $i$, then $\sigma^{-1}(i)$ is the position of the card labeled $i$. How is $\sigma(i)$ distributed if $\sigma$ has a Luce distribution? In this paper, we looked at ``typical'' values of $i$; however, what about the case $i=1$?

This is important in applications, e.g., to horse racing. Of course, if $n$ is small (e.g., $n = 5$ or $7$ in horse racing applications), the computer can blast it out. But we do not know limit theorems establishing what the limiting distribution of the label of the card at position 1 is (but recall that \cite[Section 4]{chatterjee2023enumerative} gives the distribution of the position of the card labeled 1).

Permuton limit theory does not seem very useful here. This underlines a general question:

\begin{quote}
	What functions on permutations are continuous in the permuton topology?
\end{quote}

The huge literature on probabilistic combinatorics is filled with results which give open questions for the Luce model.
One that we would love to see solved is:
\begin{quote}
	What is the limit distribution for the length of the longest increasing subsequence for a Luce-distributed permutation?
\end{quote}
The theory developed in~\cite{dubach2023locally} seems to be a good place to start.

\bibliographystyle{imsart-number} 
\bibliography{cibib,cibib2,pp}       


\end{document}